\documentclass[12pt,reqno]{amsart}

\usepackage{soul}
\usepackage{amsthm, mathtools}
\usepackage{amssymb}
\usepackage{latexsym}
\usepackage{multicol,multirow}
\usepackage{verbatim,enumerate}
\usepackage{accents}
\usepackage[usenames]{color}
\usepackage[colorlinks=true, linkcolor=blue, citecolor=green, urlcolor=blue]{hyperref}
\usepackage[nameinlink,noabbrev,capitalize]{cleveref}
\usepackage{nicematrix,makecell}
\usepackage{amsmath, amscd}
\usepackage{pifont}
\usepackage{dsfont}
\usepackage{dirtytalk,caption,subcaption}
\usepackage{tikz,setspace}
\usepackage[T1]{fontenc}
\usepackage{mathrsfs}
\usetikzlibrary{matrix,arrows,decorations.pathmorphing}
\usetikzlibrary{positioning, calc}

\advance\textwidth by 1.2in \advance\oddsidemargin by -.6in \advance\evensidemargin by -.6in

\let\oldemptyset\emptyset

\theoremstyle{plain}
\newtheorem{prop}{Proposition}
\newtheorem{thm}{Theorem}
\newtheorem{lem}{Lemma}
\newtheorem{cor}{Corollary}

\theoremstyle{definition}
\newtheorem{example}{Example}
\newtheorem{defn}{Definition}

\theoremstyle{remark}
\newtheorem{rem}{Remark}

\newcounter{cnt}
 \makeatletter
\def\mydggeometry{\makeatletter\dg@YGRID=1\dg@XGRID=20\unitlength=0.003pt\makeatother}
\makeatother \theoremstyle{remark}

\numberwithin{equation}{section}
\makeatletter
\def\section{\def\@secnumfont{\mdseries}\@startsection{section}{1}%
  \z@{.7\linespacing\@plus\linespacing}{.5\linespacing}%
  {\normalfont\scshape\centering}}
\def\subsection{\def\@secnumfont{\bfseries}\@startsection{subsection}{2}%
  {\parindent}{.5\linespacing\@plus.7\linespacing}{-.5em}%
  {\normalfont\bfseries}}
\makeatother

\begin{document}

\title[]{Marked multi-colorings and marked chromatic polynomials of hypergraphs and subspace arrangements}

\author{Chaithra Pilakkat}
\address{Technical University of Munich, TUM School of Computation, Information and Technology, Department of Mathematics, Boltzmannstr. 3, 85748 Garching bei M\"unchen, Germany.}
\email{chaithra.pilakkat@tum.de}
\thanks{C.P. acknowledges funding from the Prime Minister’s research fellowship (TF/PMRF-22-5467.03) for carrying out this research.}

\author{Shushma Rani}
\address{Department of Mathematics, Indian Institute of Technology Ropar, Rupnagar 140001, India}
\email{shushmarani@iitrpr.ac.in}
\thanks{S.R. acknowledges partial funding from the NBHM postdoctoral research fellowship (0204/33/2024/R\&D-II/14345) for carrying out this research.}

\author{Venkatesh Rajendran}
\address{Department of Mathematics, Indian Institute of Science, Bangalore 560012, India}
\email{rvenkat@iisc.ac.in}
\thanks{RV’s work is partially supported by CRG/2023/008780 of ANRF and DST FIST program-2021 [TPN700661].}

\subjclass[2010]{}
\begin{abstract}

We introduce the concepts of \emph{marked multi-colorings}, \emph{marked chromatic polynomials}, and \emph{marked (multivariate) independence series} for hypergraphs. We show that the coefficients of the \(q\)\nobreakdash-th power of the marked independence series of a hypergraph coincide with its marked chromatic polynomials in $q$, thereby generalizing a corresponding result for graphs established in \cite[Theorem~1]{CKV2025}.
These notions are then naturally extended to subspace arrangements. In particular, we prove that the number of marked multi \(q\)-colorings of a subspace arrangement is a polynomial in \(q\). We also define the (marked) independence series for subspace arrangements and prove that the \((-q)\)\nobreakdash-th power of the independence series of a hyperplane arrangement has non-negative coefficients. We further investigate the coefficientwise non-negativity of the \((-q)\)-th power of the independence series of a hypergraph. We show that all edges must have even cardinality for this non-negativity to hold, and provide a counterexample demonstrating that this condition is not sufficient.

\end{abstract}

\maketitle

\section{Motivation}

\noindent
The set of integers, nonnegative integers, positive integers, and rational numbers are denoted as $\mathbb{Z}$, $\mathbb{Z}_{\ge 0}$, $\mathbb{N}$ and $\mathbb Q$ respectively. 
For any set $S$, we define its (multi) power set, consisting of all (multi) subsets of $S$, as $P(S)$ (resp. as $P^{\mathrm{mult}}(S)$). The set of all subsets of $S$ with exactly $d$ elements is denoted by $\binom{S}{d}$. Additionally, for any $n \in \mathbb{N}$, we define the index set 
\[ \textbf{[}n \textbf{]} = \{1, 2, \dots, n\}. \]
For $n \in \mathbb{N}$, let $R_n$ and $\tilde{R}_n$ denote the polynomial ring and formal power series ring, respectively, generated by the commuting variables $x_1, \dots, x_n$ over $\mathbb{Q}$:
\[ R_n := \mathbb{Q}[x_1,\dots, x_n], \quad \tilde{R}_n := \mathbb{Q}[[x_1,\dots, x_n]].\]
For any $f \in \tilde{R}_n$ satisfying $f(0,\dots, 0) = 0$, the formal series expansion yields
\[ (1 - f)^{-1} = 1 + f + f^2 + \cdots \in \tilde{R}_n. \] We write $\mathbf{m} \geq 0$ if $m_i \geq 0$ for all $i \in \textbf{[}n \textbf{]}$. We set
$x^{\mathbf{m}} = x_1^{m_1} x_2^{m_2} \cdots x_n^{m_n},\ \text{for } \mathbf{m} \geq 0.$
Additionally, we denote the support of $\mathbf{m}$ as
\[ \mathrm{supp}(\mathbf{m}) = \{ i \in \textbf{[}n \textbf{]} : m_i \neq 0 \} \ |\mathbf{m}| =\sum_{i\in \textbf{[}n \textbf{]}} m_i, \ \text{and} \ \mathbf{m}!=\prod_{i\in \mathrm{supp}(\mathbf{m})}m_i!\]
Let $f = \sum\limits_{\mathbf{m}\ge 0}a_{\mathbf{m}}x^{\mathbf{m}}\in \tilde{R}_n$. We write $f\ge 0$ if $a_{\mathbf{m}}\ge 0$ for all $\mathbf{m}\ge 0.$

\medskip
\noindent
\textbf{Basic Questions.}
Let $n$ be a positive integer, \textit{we fix a subset $S\subseteq \textbf{[}n \textbf{]}$ and call it special}, note that $S$ could be empty.  Let $\mathcal{A} \subseteq P(\textbf{[}n \textbf{]})$ be a collection of subsets of $\textbf{[}n \textbf{]}$ such that 
\[ \text{$\emptyset \in \mathcal{A}$ and} \ \bigcup_{A \in \mathcal{A}} A = \textbf{[}n \textbf{]}. \]
\noindent
Consider the power series 
\[ I_S(\mathcal{A}, x) = \sum_{A \in \mathcal{A}} \left( \prod_{i \in A\backslash A\cap S} x_i \right)\cdot \left( \prod_{i \in A\cap S} \frac{x_i}{1-x_i} \right) \in \tilde{R}_n\]
Since $I_S(\mathcal{A}, 0) = 1$, it follows that the inverse of  $I_S(\mathcal{A}, x)$ exists in the formal power series ring $\tilde{R}_n.$ In particular, we have
$I_S(\mathcal{A}, x)^q\in \tilde{R}_n$ for any integer $q\in \mathbb{Z}.$ 
Write 
\[
I_S(\mathcal{A}, x)^q = \sum_{\mathbf{m} \in \mathbb{Z}_{\ge 0}^{n}} I_S(\mathcal{A},x)^q[x^\mathbf{m}] x^{\mathbf{m}}.
\]

\noindent
The central motivation of this work is to study the coefficients \( I_S(\mathcal{A},x)^q[x^\mathbf{m}] \), and to determine whether these coefficients have a meaningful combinatorial interpretation in terms of the collection \( \mathcal{A} \). This naturally leads to several fundamental questions:
\medskip
\begin{enumerate}
    \item What kind of functions in \( q \) are the coefficients of \( I_S(\mathcal{A}, x)^q \)?\medskip
    \item Do these coefficients have any meaningful combinatorial interpretation related to \( \mathcal{A} \), at least for certain “well-behaved” collections?\medskip
    \item When are the coefficients of $I_S(\mathcal{A}, -x)^{-1}$ non-negative?\medskip
    \item Is it possible to characterize all such \( \mathcal{A} \) for which the coefficients of $I_S(\mathcal{A}, -x)^{-1}$ are non-negative?
\end{enumerate}
\medskip
\noindent
It is straightforward to see that the coefficients of  \( I_S(\mathcal{A}, x)^q \) are polynomials in $q$ (see Proposition \ref{keyprop}).
In this paper, we provide affirmative answers to the questions (1) and (2) when \( \mathcal{A} \) forms an \textit{independence system} or \textit{abstract simplicial complex}.
Specifically, we construct a hypergraph $\mathcal{G}_{\mathcal{A}}$ naturally associated with a given independence system $\mathcal{A}$, and we show that the coefficients of 
$I_S(\mathcal{A},x)^q$ 
coincide with the
(marked) chromatic polynomials of the hypergraph $\mathcal{G}_{\mathcal{A}}$. 
To establish this connection, we introduce the notions of \emph{marked multi-colorings}, \emph{marked chromatic polynomials} and \emph{marked (multivariate) independence series/polynomials} for hypergraphs
  generalizing the respective notions that are available for graphs (see \cite{CKV2025} for more details).   

\medskip
\noindent
Let $\mathcal{G} = (\mathcal{V}, \mathcal{E})$ be a hypergraph with the non-empty vertex set $\mathcal{V}$ and edge set $\mathcal{E}$. Fix a subset $\mathcal{V}^{\mathrm{sp}}\subseteq \mathcal{V}$, and call them special vertices. A $\mathcal{V}^{\mathrm{sp}}$-marked or (simply marked) multi-coloring of $\mathcal{G}$ associated to $\mathbf{m}$ using at most $q$-colors is a map
  $f:  \mathcal{V} \rightarrow P^{\mathrm{mult}}(\{1,\dots,q\})$ satisfying $f(v)$ is a subset of $\{1,\dots,q\}$ for all $v\in  \mathcal{V}\backslash  \mathcal{V}^{\mathrm{sp}}$, 
  and $|f(v)|=m_v$ for all $v\in \mathcal{V}$ and $$ \text{$\bigcap\limits_{v\in e}\, f(v)=\emptyset$ for all $e\in \mathcal{E}.$}$$
  The marked chromatic polynomial of $\mathcal{G}$, denoted as ${_{\bold{m}}\Pi^{\mathrm{mark}}_\mathcal{G}(q)}$, counts the number of marked multi-colorings of $\mathcal{G}$ associated to $\bold{m}$  using at most $q$-colors. The marked (multivariate) independence series of $\mathcal G$ is given by
     $$I^{\mathrm{mark}}(\mathcal{G},x) = \sum_{U} x(U)$$
where $U$ runs over all marked-independent subsets $U$ of $\mathcal{G}$, i.e., $U\in P^{\mathrm{mult}}(\mathcal{V})$ whose underlying set is an independent subset of $\mathcal{G}$ and the special vertices are allowed to appear any number of times in $U$, and $x(U) = \prod_{i\in U} x_i$ (counted with multiplicity).  With these definitions, we prove the following remarkable identity for hypergraphs: 
\begin{thm}
    Let $\mathcal{G}$ be a marked hypergraph with vertex set $\mathcal{V}$ and special set of vertices $\mathcal{V}^{\mathrm{sp}}$. For $q\in \mathbb{Z}$, we have as formal series
    $$I^{\mathrm{mark}}(\mathcal{G},x)^q=\sum_{\mathbf{m}\in \mathbb{Z}_{\ge 0}^{n}} {_{\bold{m}}\Pi^{\mathrm{mark}}_\mathcal{G}(q)}\ x^\mathbf{m}$$
    where ${_{\bold{m}}\Pi^{\mathrm{mark}}_\mathcal{G}(q)}$ counts the number of marked multi-colorings of $\mathcal{G}$ associated to $\bold{m}$  using at most $q$-colors and
    $I^{\mathrm{mark}}(\mathcal{G},x)$ is the marked (multivariate) independence series of $\mathcal G$.
\end{thm} 
\noindent
Proposition~\ref{ordinarychormatic} establishes a precise relationship between marked chromatic polynomials and classical chromatic polynomials. This connection allows us to do the explicit computation of marked chromatic polynomials in the case of chordal graphs, which was first done in \cite[Theorem 2, Section 3.4]{CKV2025}. 

\medskip
\noindent
In the later part of the paper, we consider more generally subspace arrangements and study their characteristic polynomials. Generalizing the notions developed in the hypergraph setting, we introduce the concept of \(S\)-marked \(q\)\nobreakdash-colorings for a given arrangement \(\mathscr{A}\), and investigate the associated marked chromatic polynomial, which enumerates the number of distinct \(S\)\nobreakdash-marked \(q\)\nobreakdash-colorings of \(\mathscr{A}\) corresponding to \(\mathbf{m}\). It is important to emphasize that, \emph{a priori}, this function is not guaranteed to be a polynomial in \(q\). Nevertheless, we establish that the marked chromatic polynomial of \(\mathscr{A}\) is indeed a polynomial in \(q\) (see Theorem~\ref{mainthmarrangement}). We prove this by observing the connection between the marked chromatic polynomials and the characteristic polynomials of arrangements. More precisely, we have the following:
\begin{thm} Let $\mathscr{A}$ be an arrangement and  $\mathbf m = (m_1, \ldots, m_n)\in  \mathbb{Z}_{\ge 0}^n$, and let $S \subseteq \mathrm{supp}(\mathbf{m})\subseteq  \textbf{[}n \textbf{]}$ be the set of special nodes. Then we have 
   \begin{equation}\label{equarrangement}
{}_{\mathbf{m}}\Pi^{\mathrm{mark}}_{\mathscr{A}}(q) = \sum_{\boldsymbol{\lambda}\in S(\bold m)}
\frac{\chi_{\mathscr{A}(\boldsymbol{\lambda}, \mathbf{m})}(q)}
{\prod_{i\in \mathrm{supp}(\bold m)}\prod_{k=1}^{\infty}(d^{{\lambda}_i}_{k}!)}
\end{equation}
where $\chi_{\mathscr{A}(\boldsymbol{\lambda}, \mathbf{m})}(q)$ is the characteristic polynomial of the arrangement $\mathscr{A}(\boldsymbol{\lambda}, \mathbf{m}).$
In particular, ${}_{\mathbf{m}}\Pi^{\mathrm{mark}}_{\mathscr{A}}(q)$ is a polynomial in $q,$ and we can specialize $q$ to any complex number. \qed 
\end{thm}
\noindent
We refer the reader to Section \ref{mainthmarrangementsection} for the construction of $\mathscr{A}(\boldsymbol{\lambda}, \mathbf{m})$ and the definition of  $\chi_{\mathscr{A}(\boldsymbol{\lambda}, \mathbf{m})}(q)$.

\medskip
\noindent
The functions 
    ${}_{\mathbf{m}}\Pi^{\mathrm{mark}}_{\mathcal{G}}(q)$ are polynomials for hypergraphs is not very surprising as it can be derived from the general facts 
   that are recorded in Proposition \ref{keyprop}  and Theorem \ref{expmarkedindp}.
We emphasize here that there is no obvious reason to believe that the functions in $q$ counting the number of $\text{distinct $S$-marked $q-$colorings of $\mathscr{A}$}$ associated with $\mathbf m$ are polynomials in $q$, but it follows from Theorem \ref{mainthmarrangement}. For example, when we consider the hyperplane arrangement $\mathscr{A}$ with only one hyperplane $H$ whose equation is given by $x+y=z$, then the function $_{\mathbf m}\Pi_{\mathscr{A}}(q)$ (for large prime $q>>0$) counts the number of ordered tuples $(A, B, C)$ where
  $A, B, C\subseteq \mathbb F_q$ are non-empty subsets satisfying:  $$\text{ $|A|=m_1, |B|=m_2,$ $|C|=m_3$, and   $(A+B)\cap C =\emptyset$}.$$ We have the following explicit expression for 
$_{\mathbf m}\Pi_{\mathscr{A}}(q)$ using Cauchy–Davenport theorem (see Example \ref{numberthyexample} for more details):
\begin{equation}\label{examplefirstint}
        {}_{\mathbf m}\Pi_{\mathscr{A}}(q) = \sum\limits_{r=0}^{m_1m_2-m_1-m_2+1}\sum\limits_{D\subseteq \mathbb F_q \atop |D|=m_1+m_2+r-1} \alpha_D(q)
{q-(m_1+m_2+r-1)\choose m_{3}},
    \end{equation} where, for a given $D\subseteq \mathbb F_q$, the function
 $\alpha_D(q)$ counts the number of ordered tuples $(A, B)$ where
  $A, B\subseteq \mathbb F_q$ are non-empty subsets such that $|A|=m_1, |B|=m_2,$ and $A+B=D$. It is not obvious that the right hand side of the equation \ref{examplefirstint} adds up to a polynomial. We calculate it for $m_1=m_2=2$,   define $$N_\ell=\# \{(A_1,A_2): A_i\subseteq \mathbb{F}_q, |A_i|=2,\ \text{for}\ \ i=1,2, \ \text{and}\ |A_1+A_2|=\ell \}$$ by Cauchy–Davenport theorem  the possible values of $  |A_1+A_2|$ are $3$ and $4$, therefore $N_3+N_4={q\choose 2}{q\choose 2}$. If $|A+B|=3$ then $A $ and $B$ are in arithmetic
progression with the same diﬀerence by Vosper's theorem (see \cite{Vosper}). Therefore we have $N_3=\frac{q^2(q-1)}{2}$ and $N_4={q\choose 2}{q\choose 2}-\frac{q^2(q-1)}{2}=\frac{q^2(q-1)(q-3)}{4}$. Then we have 
$$ _{\mathbf m}\Pi_{\mathscr{A}}(q) = \frac{q^2(q-1)}{2} {q-3\choose m_3}+\frac{q^2(q-1)(q-3)}{4}{q-4\choose m_3}.$$

\medskip
\noindent
In the final part of the paper, we investigate the non-negativity of the coefficients of the inverse of the independence polynomial of a hypergraph. Given a hypergraph \(\mathcal{G} = (\mathcal{V}, \mathcal{E})\), its independence polynomial is defined as
\[
I(\mathcal{G}, x) = \sum_{U} x(U),
\]
where \(x(U) = \prod_{i \in U} x_i\), and the sum runs over all independent subsets \(U \subseteq \mathcal{V}\), that is, those subsets which do not contain any edge \(e \in \mathcal{E}\) as a subset. We prove that the coefficients of \(I(\mathcal{G}, -x)^{-1}\) are non-negative when \(\mathcal{G}\) is a graph (i.e., all edges have cardinality 2). Furthermore, we show that the presence of only even-cardinality edges is a necessary condition for this non-negativity to hold in the general hypergraph setting.
However, through an explicit counterexample, we demonstrate that this condition alone is not sufficient.


\section{Basic expressions using binomial series}
\noindent First, we derive an easy combinatorial expression for the coefficients \( I_S(\mathcal{A},x)^q[x^\mathbf{m}] \) using the binomial series expansion. Although this is standard, we include it here for the reader’s convenience.

\begin{defn}
Let $\mathbf{m} = (m_i)_{i \in \textbf{[}n \textbf{]}} \in \mathbb{Z}_{\ge 0}^n$ be a tuple of nonnegative integers. We define $P_k^S(\mathcal{A}, \mathbf{m})$ as the set of all ordered $k$-tuples $(A_1, \dots, A_k)$ such that: 
\begin{enumerate}
\item we have $A_j \in P^{\mathrm{mult}}(\textbf{[}n \textbf{]})\backslash \{\emptyset\}$ for all $1\le j\le k$,
    \item the underlying set of $A_j$ belongs to  $\mathcal{A}$,
    \item the elements from $\textbf{[}n \textbf{]}\backslash S$ appear at most once in $A_j$ for each $1\le j\le k$,
    \item the disjoint union of $A_1, \dots, A_k$ corresponds to the multiset $\{ i, \dots, i : i \in \mathrm{supp}(\mathbf{m}) \}$, where each element $i$ appears exactly $m_i$ times.
\end{enumerate}
\end{defn}
\noindent Now we derive an explicit expression for the coefficients of
$ I_S(\mathcal{A}, x)^q$  in terms of $P_k^S(\mathcal{A}, \mathbf{m})$ using the binomial series expansion.

\begin{prop}\label{keyprop}
For all $\mathbf{m} \geq 0$, we have
\[I_S(\mathcal{A},x)^q[x^\mathbf{m}] = \sum_{k \geq 0}\binom{q}{k} |P_k^S(\mathcal{A}, \mathbf{m})|. \]
In particular,  for a fixed $\mathbf{m} \geq 0$, we have that $I_S(\mathcal{A},x)^q[x^\mathbf{m}]$ is a polynomial in $q.$
\end{prop}
\begin{proof}
  For any $f\in \tilde R_n$ with constant term $1$, we have 
    $$f^q=\sum_{k\geq 0}\binom{q}{k}(f-1)^k, \ \  \text{Where}\  \binom{q}{k}=\frac{q(q-1)\dots(q-(k-1))}{k!}$$
Let us set $f=I_S(\mathcal{A},x)$ and note that 
    $$f-1 = I_S(\mathcal{A},x)-1=\sum_{A\in \mathcal{A}\setminus \{\oldemptyset\}} \tilde{x}(A)$$
where \[
\widetilde{x}(A)
=
\left(\prod_{i\in A\setminus S}x_i\right)
\left(\prod_{r\in A\cap S}\frac{x_r}{1-x_r}\right)
=
\left(\prod_{i\in A\setminus S}x_i\right)
\left(
\sum_{\substack{(p_r)_{r\in A\cap S}\\ p_r\geq 1}}
\prod_{r\in A\cap S}x_r^{p_r}
\right).
\] For $\mathbf{m}\ge 0$, the coefficient of $(f-1)^k[x^\mathbf{m}]$ is $\sum\limits_{(A_1,\dots,A_k)} 1$, where
 the sum ranges over all $k$-tuples $(A_1,\dots,A_k)$ of multi-subsets of $\textbf{[}n \textbf{]}$ satisfying
    \begin{enumerate}
\item the underlying set of $A_r$ belongs to $\mathcal{A}\setminus\{\oldemptyset\}$ for all $1\leq r\leq k$,
    \item the elements from $\textbf{[}n \textbf{]}\backslash S$ appear at most once in $A_j$ for each $1\le j\le k$,
        \item the disjoint union of $\dot{\bigcup}_{r=1}^{k} A_r$ is  equal to the multiset $\{i,\dots,i:i\in\mathrm{supp}(\mathbf{m})\},$ where $i$ appears exactly $m_i$ number of times for each $i\in \mathrm{supp}(\mathbf{m})$.
    \end{enumerate}
    It follows that $(A_1,\dots,A_k)\in P_k^S(\mathcal{A}, \mathbf{m})$ and the sum ranges over all the elements from $P_k^S(\mathcal{A}, \mathbf{m})$. 
    Thus we have 
    $$(f-1)^k[x^\mathbf{m}]=|P_k^S(\mathcal{A}, \mathbf{m})|.$$
    Hence we have 
    $$I_S(\mathcal{A},x)^q[x^\mathbf{m}]=\sum_{k\geq 0}\binom{q}{k}(f-1)^k[x^\mathbf{m}]=\sum_{k\geq 0}\binom{q}{k}|P_k^S(\mathcal{A}, \mathbf{m})|.$$ This completes the proof.
\end{proof}

\noindent
\begin{rem}
Proposition~\ref{keyprop}, in its current form, is unsatisfactory for several reasons. However, when \(\mathcal{A}\) is a collection of subsets of \([n]\) forming an \emph{independence system}, one can construct a hypergraph \(\mathcal{G}_{\mathcal{A}}\) such that \(\mathcal{A}\) coincides with the set of independent subsets of \(\mathcal{G}_{\mathcal{A}}\). In this setting, the coefficients of \(I(\mathcal{A}, x)^q\) admit a combinatorial interpretation as the marked chromatic polynomials of the associated hypergraph \(\mathcal{G}_{\mathcal{A}}\). This construction is developed in the first part of the paper.
 
\end{rem}

\medskip
\noindent
\textbf{Observation.} Let $f(x)$ and $g(x)$ be two given polynomials over complex numbers, and assume that $f(q) = g(q)$ for infinitely many complex numbers $q$. Then we must have 
$f(q) = g(q)$ for all complex numbers $q$, and also we have the equality $f(x) = g(x)$ as polynomials. This will be used throughout without further mention.

\section{Hypergraphs and marked chromatic polynomials}
\noindent
In this section, we recall some basic definitions of hypergraphs. 

\subsection{}
Let $n$ be a positive integer. The pair $\mathcal{G} = (\mathcal{V}, \mathcal{E})$ is called a \textit{hypergraph}, 
where we take the set of \textit{vertices} $\mathcal{V} = \textbf{[}n \textbf{]}$ and the set of \textit{edges} is just a collection of subsets of $\textbf{[}n \textbf{]}$ (possibly empty collection), i.e., $\mathcal{E}\subseteq P(\textbf{[}n \textbf{]})$. 

\begin{defn}
    \begin{enumerate}
        \item  A hypergraph $\mathcal{G}$ is called \textit{simple} if for any $e, f \in \mathcal{E}$, we have $|e| \geq 2$ and $e \subseteq f$ implies $e = f$.
       \medskip
       \item A hypergraph $\mathcal{G}$ is called \textit{graph} if $|e|=2$ for all $e\in \mathcal{E}$. Note that if 
      $\mathcal{G}$ is a simple graph then it  has no loops and multiple edges. 
        \medskip
        \item For any non-empty subset $U \subseteq \mathcal{V}$, the \textit{induced hypergraph} on $U$, denoted $\mathcal{G}_U$, is the sub-hypergraph with vertex set $\mathcal V(\mathcal{G}_U) = U$ and edge set $\mathcal  E(\mathcal{G}_U)$ consisting of all edges of $\mathcal{G}$ that are entirely contained in $U$.\medskip
        \item  Two distinct vertices $u, v \in \mathcal{V}$ are called \textit{neighbors} in $\mathcal{G}$ if there exists an edge $e \in \mathcal{E}$ such that $u, v \in e$. For any vertex $v \in \mathcal{V}$, the \textit{neighborhood} of $v$ is defined as
\[ N_\mathcal{G}(v) = \{ u \in \mathcal{V} \mid u \text{ is a neighbor of } v \}. \]
If $N_\mathcal{G}(v) = \emptyset$, then $v$ is called an \textit{isolated vertex}. The \textit{closed neighborhood} of $v$ is given by $N_\mathcal{G}[v] = N_\mathcal{G}(v) \cup \{v\}$.\medskip
\item  A subset $I \subseteq \mathcal{V}$ is called \textit{independent} if no edge of $\mathcal{G}$ is entirely contained in $I$, i.e., $e \not\subseteq I$ for all $e \in \mathcal{E}$. If $\mathcal{G}$ is a simple hypergraph, then both $\emptyset$ and singleton sets $\{v\}$ for $v \in \mathcal{V}$ are independent subsets. The set of all independent subsets of $\mathcal{G}$ is denoted as $I(\mathcal{G})$.

    \end{enumerate}
\end{defn}

\begin{defn}
The \textit{(multivariate) independence polynomial} of the hypergraph $\mathcal{G}$ is defined as
\[ I(\mathcal{G}, x) = \sum_{I \in I(\mathcal{G})} \prod_{v \in I} x_v.\] Note that $I(\mathcal{G}, x)$ is a polynomial in the variables $\{x_v \mid v \in \mathcal{V}\}$. Additionally, for any integer $q \in \mathbb{Z}$,  $I(\mathcal{G}, x)^q$ belongs to $\tilde{R}_n$.
\end{defn}

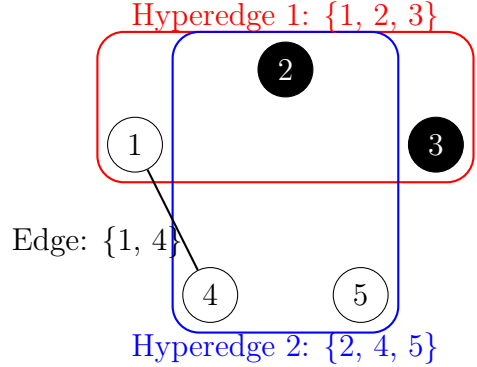
\begin{figure}[h]
    \centering
    \begin{tikzpicture}
        \node[draw, circle] (1) at (0,2) {1};
        \node[draw, circle, fill=black, text=white] (2) at (2,3) {2};
        \node[draw, circle, fill=black, text=white] (3) at (4,2) {3};
        \node[draw, circle] (4) at (1,0) {4};
        \node[draw, circle] (5) at (3,0) {5};
        
        \draw[rounded corners=10pt, red, thick] (-0.5,1.5) -- (4.5,1.5) -- (4.5,3.5) -- (-0.5,3.5) -- cycle;
        \draw[rounded corners=10pt, blue, thick] (0.5,-0.5) -- (3.5,-0.5) -- (3.5,3.5) -- (0.5,3.5) -- cycle;
        
        \draw[thick] (1) -- (4);
    
        \node[red] at (2,3.7) {Hyperedge 1: \{1, 2, 3\}};
        \node[blue] at (2,-0.7) {Hyperedge 2: \{2, 4, 5\}};
        \node at (-0.5,0.7) {Edge: \{1, 4\}};
    \end{tikzpicture}
    \caption{Hypergraph with hyperedges and a standard edge; the vertices 2 and 3 are special}
    \label{fig:hypergraph1}
\end{figure}

\subsection{} We choose $\mathcal{V}^{\mathrm{sp}}\subseteq \mathcal{V}$ and declare them as special vertices, and we allow the special vertices to repeat as many times as possible (but only finitely many times) in 
independent sets. This gives us a new collection of marked-independent subsets of $\mathcal{G}$. More precisely, we have the following:

\begin{defn}
 An element $U\in P^{\mathrm{mult}}(\mathcal{V})$ is called \textit{marked-independent} if
    \begin{enumerate}
        \item     the underlying set of $U$ is independent, equivalently $e\nsubseteq U$ for any $e\in \mathcal{E}$,
\item  the vertices from $\mathcal{V}\backslash \mathcal{V}^{\mathrm{sp}}$ appear at most once in $U$.
   \end{enumerate}
\end{defn}

\begin{defn}
 \begin{enumerate}
\item We denote by $I^{\mathrm{mark}}({\mathcal{G}})$ the set of all marked-independent subsets of $\mathcal{G}$. We understand $\emptyset\in I^{\mathrm{mark}}(\mathcal{G})$. If $\mathcal{G}$ is a simple hypergraph then we have 
$\{v\}\in I^{\mathrm{mark}}(\mathcal{G})$ for all $v\in \mathcal{V}$ and $\{v, \dots, v\}\in I^{\mathrm{mark}}(\mathcal{G})$, where $v$ appears $m$-times, for all $m\ge 0$ and $v\in \mathcal{V}^{\mathrm{sp}}.$
\medskip
\item Given $U\in I^{\mathrm{mark}}(\mathcal{G})$, set $x(U) = \prod_{v\in \mathcal{V}}x_v^{r_v}$, where $v\in \mathcal{V}$ appears $r_v$ times in $U$.  
The marked (multivariate) independence series of the hypergraph $\mathcal{G}$ is defined as
$$I^{\mathrm{mark}}(\mathcal{G},x)=\sum_{U\in I^{\mathrm{mark}}(\mathcal{G})} x(U).$$
When $\mathcal{V}^{\mathrm{sp}} =\emptyset$, then we have $I^{\mathrm{mark}}(\mathcal{G},x) =I(\mathcal{G},x)$.
\end{enumerate}
\end{defn}

\begin{example}
Consider the hypergraph in Figure \ref{fig:hypergraph1}. The vertices $\mathcal{V}$ are $\{1, 2, 3, 4, 5\}$ and the special vertices are $\{2, 3\}$. 
The edges are $\mathcal{E} = \{e_1, e_2, e_3\}$ where $e_1 = \{1, 2, 3\}$, $e_2 = \{2, 4, 5\}$, and $e_3 = \{1, 4\}$. Note that 
$e_1$ and $e_2$ are hyperedges, and $e_3$ is a standard edge. 

\begin{align*}
    I^{\mathrm{mark}}(\mathcal{G},x) = & \ 1 +x_1 +\frac{x_2}{1-x_2} +\frac{x_3}{1-x_3} + x_4 + x_5 
     + x_1\frac{x_2}{1-x_2} + x_1\frac{x_3}{1-x_3}\\ & + x_1x_5 + \frac{x_2}{1-x_2}\frac{x_3}{1-x_3} 
    + \frac{x_2}{1-x_2}x_4 + \frac{x_2}{1-x_2}x_5 + \frac{x_3}{1-x_3}x_4 \\ & + \frac{x_3}{1-x_3}x_5 + x_4x_5 +x_1\frac{x_2}{1-x_2}x_5+\frac{x_2}{1-x_2}\frac{x_3}{1-x_3}x_5+
    x_1\frac{x_3}{1-x_3}x_5\\&+\frac{x_2}{1-x_2}\frac{x_3}{1-x_3}x_4+\frac{x_3}{1-x_3}x_4x_5.
\end{align*}

\end{example}

\medskip

\section{Marked colorings and marked chromatic polynomials}

\subsection{}
We generalize the vertex coloring of hypergraphs as in \cite{CKV2025} by allowing the colors to repeat for special vertices. We call this new coloring of hypergraphs as marked coloring. 

\begin{defn} Given $q\in\mathbb{N}$ we define the following.
\begin{enumerate}
        \item We call a map $_{\mathbf{m}}\Gamma^{\mathrm{mark}}_\mathcal{G}:  \mathcal{V} \rightarrow P^{\mathrm{mult}}(\{1,\dots,q\})$ a $\mathcal{V}^{\mathrm{sp}}$-marked or (simply marked) multi-coloring of $\mathcal{G}$ associated to $\bold{m}$ using at most $q$-colors if the following conditions are satisfied:\vspace{0,1cm}
        
        \begin{enumerate}[(i)]
            \item $_{\bold{m}}\Gamma^{\mathrm{mark}}_\mathcal{G}(v)$ is a subset of $\{1,\dots,q\}$ for all $v\in  \mathcal{V}\backslash  \mathcal{V}^{\mathrm{sp}}$,\vspace{0,1cm}
            \item for all $v\in \mathcal{V}$ we have $|_{\mathbf{m}}\Gamma^{\mathrm{mark}}_\mathcal{G}(v)|=m_v$,\vspace{0,1cm}

             \item for all $e\in \mathcal{E}$, we have that $\bigcap\limits_{v\in e}\, _{\mathbf{m}}\Gamma^{\mathrm{mark}}_\mathcal{G}(v)=\emptyset$.
            
            
        \end{enumerate}
         \vspace{0.1cm}
        \medskip
        \item We say that two given $_{\bold{m}}\Gamma^{\mathrm{mark}}_\mathcal{G}$ and $_{\bold{m}}\Gamma'^{\mathrm{mark}}_\mathcal{G}$ proper vertex marked multi-coloring  of $\mathcal{G} $ associated to $\bold{m}$ using at most $q$-colors \textit{distinct} if
        $_{\bold{m}}\Gamma^{\mathrm{mark}}_\mathcal{G}(v)\neq _{\bold{m}}\Gamma'^{\mathrm{mark}}_\mathcal{G}(v)$ 
        for some $v\in \mathcal{V}$ (that is if they differ at some vertex  $v\in \mathcal{V}$ ).
        \item The number of marked multi-colorings of $\mathcal{G}$ associated to $\bold{m}$  using at most $q$-colors 
        is denoted by $_{\bold{m}}\Pi^{\mathrm{mark}}_\mathcal{G}(q)$. We will see that $_{\bold{m}}\Pi^{\mathrm{mark}}_\mathcal{G}(q)$ is a polynomial in $q$ and we call it as marked-chromatic polynomial of $\mathcal{G.}$
    \end{enumerate}
   
\end{defn}

\medskip
\noindent
\textbf{
Notation:} When $\mathbf{m} = (1, \ldots, 1)$, then we simply use 
        $\Pi^{\mathrm{mark}}_\mathcal{G}(q)$ to denote the chromatic polynomial $_{\mathbf{1}}\Pi^{\mathrm{mark}}_\mathcal{G}(q)$. If the set of special vertices $\mathcal{V}^{\mathrm{sp}}$ is empty then we drop "mark" and
        use $_{\bold{m}}\Pi_\mathcal{G}(q)$ instead of $_{\bold{m}}\Pi^{\mathrm{mark}}_\mathcal{G}(q)$ to indicate that it refers to usual vertex multi-coloring. In particular,
        $\Pi_\mathcal{G}(q)$ denotes the usual chromatic polynomial of $\mathcal{G}.$

\subsection{}
We have the following definition of partitions of $\mathcal{G}$ into marked-independent sets. 
\begin{defn}
We denote by ${P}_k^{\mathrm{mult}}(\mathcal{G}, \mathbf{m})$ the set of all ordered $k$-tuples $(P_1,\dots,P_k)$ satisfying:
\begin{enumerate}
\item $P_r$ is a non-empty multi-set and $P_r\in I^{\mathrm{mark}}({\mathcal{G}})$ for all $1\leq r\leq k$, \medskip
    \item the disjoint union (as multisets) of $P_1,\dots, P_k$ is equal to the multiset
    $$\{\underbrace{v, \dots, v}_{m_v\text{-times}} : v\in\mathrm{supp}(\bold{m})\}.$$
\end{enumerate}
We can write the number of marked multi-colorings  of $\mathcal{G}$ associated to $\bold{m}$  using at most $q$-colors as follows:
\end{defn}
\begin{prop}\label{props6}
For each $q\in \mathbb{N}$ and $\mathbf{m}\in\mathbb{Z}_{\ge 0}^n$ with non-empty support, we have
    \begin{equation}
    \label{equation1}
    _{\bold{m}}\Pi^{\mathrm{mark}}_\mathcal{G}(q)=\sum_{k\geq 1}| {P}_k^{\mathrm{mult}}(\mathcal{G}, \mathbf{m})| \binom{q}{k}.
\end{equation} 
In particular, $_{\bold{m}}\Pi^{\mathrm{mark}}_\mathcal{G}(q)$ is a polynomial in $q$ and we can specialize $q$ as any complex number. 
\end{prop}
\begin{proof}

We fix a subset of colors \(\{c_1 < \dots < c_k\} \subseteq \{1, \dots, q\}\) and consider marked multi-colorings \({}_{\mathbf{m}}\Gamma^{\mathrm{mark}}_\mathcal{G}\)
whose set of colors used is exactly \(\{c_1,\ldots,c_k\}\). 
Given such a marked multi-coloring \({}_{\mathbf{m}}\Gamma^{\mathrm{mark}}_\mathcal{G}\), we can naturally associate an element \((P_1, \dots, P_k) \in P_k^{\mathrm{mult}}(\mathcal{G}, \mathbf{m})\) as follows. Let \(a_{v, r}\) denote the multiplicity of color \(c_r\) in the coloring of vertex \(v\), i.e., in \({}_{\mathbf{m}}\Gamma^{\mathrm{mark}}_\mathcal{G}(v)\). We then define
\[
P_r = \{v^{a_{v,r}} : v \in \mathcal{V}\}, \quad 1 \le r \le k.
\]
\noindent
By construction, each \(P_r\) is a non-empty multiset, and the collection \(\{P_r\}_{r=1}^k\) forms a partition of the multiset \(\{v^{m_v} : v \in \mathrm{supp}(\mathbf{m})\}\). Moreover, it is evident from the definition that \(a_{v,r} \le 1\) for any non-special vertex \(v \in \mathcal{V} \setminus \mathcal{V}^{\mathrm{sp}}\).

\medskip
\noindent
Now, suppose \(P_r \notin I^{\mathrm{mark}}(\mathcal{G})\). Then the underlying set of \(P_r\) must be non-independent; that is, there exists an edge \(e \in \mathcal{E}\) such that \(e \subseteq P_r\). This implies that all vertices in \(e\) receive the same color \(c_r\), contradicting the definition of the marked multi-coloring \({}_{\mathbf{m}}\Gamma^{\mathrm{mark}}_\mathcal{G}\). Hence, each \(P_r\) belongs to \(I^{\mathrm{mark}}(\mathcal{G})\), and we conclude that \((P_1, \dots, P_k) \in P_k^{\mathrm{mult}}(\mathcal{G}, \mathbf{m})\). Conversely, given any multi-partition \((P_1, \dots, P_k) \in P_k^{\mathrm{mult}}(\mathcal{G}, \mathbf{m})\), we assign the color \(c_r\) to the vertices in \(P_r\), counted with multiplicity, for each \(1 \le r \le k\).
For any $e=\{v_{i_1}, \ldots, v_{i_p}\} \in \mathcal E$, note that $e \nsubseteq P_j $ for any $1 \leq j \leq k$, i.e., all vertices of $e$ can't receive the same color. Thus $\bigcap\limits_{ v \in e} {}_{\mathbf m} \Gamma^{\mathrm{mark}}_{\mathcal G}(v)= \emptyset. $
This yields a marked multi-coloring of \(\mathcal{G}\), and the correspondence described above is bijective.

\end{proof}

\medskip
\noindent
The following is the main result, which gives the interpretation of $I^{\mathrm{mark}}(\mathcal{G},x)^q[x^\mathbf{m}]$ in terms of 
marked-chromatic polynomials. 
\begin{thm}\label{expmarkedindp}
    Let $\mathcal{G}$ be a marked hypergraph with vertex set $\mathcal{V}$ and special set of vertices $\mathcal{V}^{\mathrm{sp}}$. For $q\in \mathbb{Z}$, we have as formal series
    $$I^{\mathrm{mark}}(\mathcal{G},x)^q=\sum_{\mathbf{m}\in \mathbb{Z}_{\ge 0}^{n}} {_{\bold{m}}\Pi^{\mathrm{mark}}_\mathcal{G}(q)}\ x^\mathbf{m}.$$
\end{thm} 
\begin{proof} 
    For any $f\in \tilde{R}_n$ with constant term $1$, we have 
    $$f^q=\sum_{k\geq 0}\binom{q}{k}(f-1)^k,$$
    Set $f=I^{\mathrm{mark}}(\mathcal{G},x)$ and note that 
    $$f-1 = I^{\mathrm{mark}}(\mathcal{G},x)-1=\sum_{U\in I^{\mathrm{mark}}(\mathcal{G})\setminus \{\emptyset\}} \prod_{v\in U}x_v.$$
    Thus, for $k\geq 1$ and $\mathbf{m}\neq 0$, the coefficient $(f-1)^k[x^\mathbf{m}]$ is given by the ordered tuples $(U_1, U_2, \ldots, U_k)$ satisfying the following conditions:
    \begin{enumerate}
        \item $U_r$ is multi-set and $U_r \in I^{\mathrm{mark}}(\mathcal{G})\setminus \{\emptyset\}$ for all $1 \leq r \leq k;$
        \item the disjoint union (as multisets) of $U_1, \ldots, U_k$ is equal to the multiset $$\{\underbrace{v, \ldots,v}_{m_v\text{-times}} : v\in\mathrm{supp}(\bold{m})\}.$$
    \end{enumerate}
   the cardinality of such ordered tuples $(U_1,\dots ,U_k)$ satisfying the above two conditions is $|P_k^{\mathrm{mult}}(\mathcal{G}, \mathbf{m})|$,  and we get with Proposition~\ref{props6}:
$$I^{\mathrm{mark}}(\mathcal{G},x)^q[x^\mathbf{m}]=\sum_{k\geq 1}\binom{q}{k}(f-1)^k[x^\mathbf{m}]=\sum_{k\geq 1}\binom{q}{k}|P_k^{\mathrm{mult}}(\mathcal{G}, \mathbf{m})|={_{\bold{m}}\Pi^{\mathrm{mark}}_\mathcal{G}(q)},\  \text{for all $\mathbf{m}\neq 0$}.$$ Since the above equality holds for all $q\in\mathbb Z$, we have $I^{\mathrm{mark}}(\mathcal{G},x)^q[x^\mathbf{m}]={_{\bold{m}}\Pi^{\mathrm{mark}}_\mathcal{G}(q)}$
as polynomials in $q.$
This completes the proof; the coefficients for $\mathbf{m}=0$ clearly coincide.
\end{proof}

\begin{example}\label{impexample}
Consider a hypergraph $\mathcal G$ with vertex set $\mathcal V=\textbf{[}n \textbf{]}$ and the edge set $\mathcal E=\{ e\}$ where $e=\textbf{[}n \textbf{]}$, and set $\mathcal V^{\mathrm{sp}}=\emptyset$. Then the independence polynomial of the hypergraph $\mathcal{G}$ is
    $$ I(\mathcal G, x)=  
    \sum\limits_{\substack{A\subseteq \textbf{[}n \textbf{]}\\ A\neq \textbf{[}n \textbf{]}}} \left(\prod_{i\in A}x_i\right) =  \left(\prod\limits_{i \in \textbf{[}n\textbf{]}}(1+x_i)\right)- (x_1 x_2 \ldots x_n)$$
    Set $\zeta:= \prod\limits_{i \in \textbf{[}n \textbf{]}}(1+x_i).$ Then for $q\in \mathbb Z_{\ge 0}$,  we have
    $$\zeta^q= \sum\limits_{0 \leq a_1, \dots, a_n \leq q}  \binom{q}{a_1} \binom{q}{a_2} \ldots \binom{q}{a_n} x_1^{a_1} x_2^{a_2} \ldots x_n^{a_n}$$ and $$
      I(\mathcal G, x)^q=  (\zeta- x_1 x_2 \ldots x_n)^q
      =  \sum\limits_{k=0}^{q} (-1)^k \binom{q}{k} \zeta^{q-k} (x_1 x_2 \ldots x_n)^k$$ $$
  \begin{aligned}   I(\mathcal G, x)^q \,[x_1^{m_1} x_2^{m_2} \ldots x_n^{m_n}] = &  \sum\limits_{k=0}^{q} (-1)^k \binom{q}{k} \zeta^{q-k}[x_1^{m_1-k} x_2^{m_2-k} \ldots x_n^{m_n-k}]\\
      = & \sum\limits_{k=0}^{\min\{m_1, m_2, \ldots, m_n\}} (-1)^k \binom{q}{k} \binom{q-k}{m_1-k} \binom{q-k}{m_2-k}
 \ldots \binom{q-k}{m_n-k}    \end{aligned}$$
 Thus we get $$ {_{\bold{m}}\Pi^{\mathrm{mark}}_\mathcal{G}(q)}=\sum\limits_{k=0}^{\min\{m_1, m_2, \ldots, m_n\}} (-1)^k \binom{q}{k} \binom{q-k}{m_1-k} \binom{q-k}{m_2-k}
 \ldots \binom{q-k}{m_n-k} $$ as a corollary using Theorem \ref{expmarkedindp}.
\end{example}
\subsection{}  Let $G =(V, E)$ be a simple graph with vertex set $\textbf{[}n \textbf{]}$ and edge set $E.$

\begin{figure}[ht]
    \centering
    \begin{tikzpicture}
    \tikzstyle{B}=[circle,draw=black!80,fill=black!80,thick]
    \node[B] (1) at (0,0) [label=below:$n-1$]{};
    \node[B] (2) at (1.5,0) [label=below:$3$]{};
    \node[B] (3) at (-0.8,1.5) [label=left:$n$]{};
    \node[B] (4) at (2.3,1.5) [label=right:$2$]{};
    \node[B] (5) at (0.75,2.7) [label=right:$1$]{};
    \path[-] (2) edge node[left]{} (4);
    \path[-] (1) edge node[left] {} (3);
    \path[-] (4) edge node[left] {} (5);
    \path[-] (3) edge node[left] {} (5);
    \path (1) -- node[auto=false]{\ldots} (2);
\end{tikzpicture}
    \caption{Cycle $C_n$}
    \label{C_n}
\end{figure}
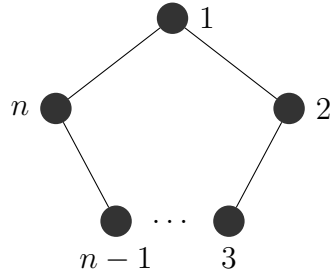
\begin{defn}
   We say that $G$ is \textit{chordal} if it has no induced subgraph that is isomorphic to $C_m$, $m\ge 4,$ where
   $C_m$ is the cycle graph with $m$-vertices.    It is a well-known fact that chordal graphs possess a perfect elimination ordering on the vertices. An ordering $\textbf{[}n \textbf{]} = \{1, \dots, n\}$ is called \textit{perfect elimination ordering} if for each $1\le k\le n$, we have the subgraph spanned by
$V_k = \{i : 1\le i\le k\}\cap N_G[k]$ must be a clique (complete subgraph).
\end{defn}
\begin{example}
Let $G$ be a chordal graph with the vertex set $\textbf{[}n \textbf{]}$, let the ordering on the vertices $\textbf{[}n \textbf{]} = \{1, \ldots, n\}$ be a perfect elimination ordering of the vertices of $\mathcal{G}$.
For $\mathbf{m}\ge 0$
with $\mathrm{supp}(\mathbf{m}) = \{i_1 < \cdots < i_N\}$, induced from the perfect elimination ordering of $\textbf{[}n \textbf{]}$, we have
 \[
{_{\mathbf{m}}\Pi_G(q)} =\prod\limits_{r=1}^{N} \binom{q-a_r(\mathbf{m})+m_{i_r}}{m_{i_r}}
\]
where $a_r(\mathbf{m})=m_{i_r}+\sum\limits_{\substack{1\le s <r\\ i_s\in N_G(i_r)}} m_{i_s}$,
for $1\le r \le N$. 

\end{example}

\begin{example} For the cycle graph $C_n$, we have the following formula by Read (see \cite{Read})
\medskip
$$   {_{\mathbf{m}}\Pi_{C_n}(q)}= \left(\prod_{r=1}^{n}(q)_{(m_r+m_{r+1})}\right)\times
    \left(\sum\limits_{k=0}^n\Big((-1)^{kn}v_k(q)\prod_{i=1}^n\frac{(m_i)^k}{(q)^{m_i+k}}\Big)\right)
$$
where $m_{n+1}=m_1$ 
and $v_k(q)=\binom{q}{k}-\binom{q}{k-1}$ and $(q)_k=q(q-1)(q-2)\cdots(q-k+1)$.
\end{example}

\section{Marked vs ordinary chromatic polynomials}
\noindent
Let $\mathcal{G} = (\mathcal{V}, \mathcal{E})$ be a \textit{hypergraph} as before with \textit{vertices} $\mathcal{V} = \textbf{[}n \textbf{]}$ and special vertices $\mathcal{V}^{\mathrm{sp}}.$ There is a close relationship between marked-chromatic polynomials and ordinary chromatic polynomials of graphs, and this was first observed for chordal graphs in \cite[Theorem 2, Section 3.4]{CKV2025}).
In this section, we generalize this to all hypergraphs. 

  \begin{defn} For a non-negative integer $m$, a partition $\lambda$ of $m$, denoted as $\lambda \vdash m$, is an ordered tuple $\lambda = (\lambda_1\ge \cdots \ge \lambda_k>0)$ of positive integers such that $\sum_{i=1}^k \lambda_i =m$, and the length of the partition $\ell(\lambda)$ is defined to be $k$. For any $j\in \mathbb{N}$, let $d^\lambda_{j}$ be the number appearances of $j$ in the partition $\lambda$.

\medskip
\noindent
Given a tuple $\bold{m} = (m_i)_{i \in  \textbf{[}n \textbf{]}}$ of non-negative integers, we define $$S(\bold m)=\{ \boldsymbol{\lambda}=(\lambda_i)_{i \in  \textbf{[}n \textbf{]}}: \lambda_i \text{ is a partition of }m_i \text{ and } \lambda_i=(1^{m_i})\ \forall i\notin \mathcal{V}^{\mathrm{sp}}\}.$$
For each tuple $\boldsymbol{\lambda}\in S( \bold m)$, we associate a hypergraph $\mathcal{G}(\boldsymbol{\lambda}, \bold m)$ as follows:
\begin{itemize}
\item For each $i \in  \textbf{[}n \textbf{]}$, replace the vertex $i$ with a clique of size $\ell(\lambda_i)$ (i.e., a complete graph of $\ell(\lambda_i)$ vertices).
\item For each hyperedge $\{i_1, \ldots, i_k\} \in \mathcal E$, include in $\mathcal{G}(\boldsymbol{\lambda},\boldsymbol{m})$  hyperedges consisting of one vertex chosen from each of the cliques corresponding to $i_1, \ldots, i_k$.
\end{itemize}
\end{defn}
\noindent
Note that ${\lambda}_i$ is the empty partition for all $i\in  \textbf{[}n \textbf{]}$ with $m_i=0$.
We have the following relation between marked chromatic polynomials of hypergraph and the  ordinary chromatic polynomials of hypergraphs. 

\begin{prop}\label{ordinarychormatic} We have 
   \begin{equation}
    \label{equation}
    {_{\bold{m}}\Pi^{\mathrm{mark}}_\mathcal{G}(q)}=\sum\limits_{\boldsymbol{\lambda}\in S(\bold m)} \frac{\Pi_{\mathcal{G}(\boldsymbol{\lambda}, \mathbf{m})}(q)}{\prod_{i\in \mathrm{supp}(\bold m)}\prod_{k=1}^{\infty}(d^{{\lambda}_i}_{k}!)}
\end{equation} 
\end{prop}
\begin{proof}
    Let ${_{\bold{m}}\Gamma^{\mathrm{mark}}_\mathcal{G}}$ be a marked multi-coloring of the hypergraph $\mathcal{G}$ associated to $\bold{m}$ using at most $q$ colors. Then by the definition, $_{\bold{m}}\Gamma^{\mathrm{mark}}_\mathcal{G}(i)$ is a multi-subset (resp. subset) of $\{1,\dots,q\}$ if $i\in \mathcal{V}^{\mathrm{sp}}$ (resp. $i\notin \mathcal{V}^{\mathrm{sp}}$)  with $|{_{\bold{m}}\Gamma^{\mathrm{mark}}_\mathcal{G}}(i)|=m_i$.
     To each element $\boldsymbol{\lambda} = ({\lambda}_i)_{i\in  \textbf{[}n \textbf{]}} \in S(\mathbf{m})$, we associate a family of ordinary vertex colorings of the hypergraph $\mathcal{G}(\boldsymbol{\lambda}, \mathbf{m})$ as follows.
For any index $i\in  \textbf{[}n \textbf{]}$, if $m_i = 0$, we set ${\lambda}_i$ as empty partition. Otherwise, if $i \in \mathrm{supp}(\mathbf{m})$, we consider the  multiset
\[
{_{\bold{m}}\Gamma^{\mathrm{mark}}_\mathcal{G}}(i) = \{1^{b_{i,1}}, \dots, q^{b_{i,q}}\},
\]
where  $b_{i,1}, \dots, b_{i,q}$ are non-negative integers and sum to $m_i$. Then ${\lambda}_i = (\lambda^1_i \ge \cdots \ge \lambda^{\ell({\lambda}_i)}_i > 0)$ is defined to be the unique partition of $m_i$ corresponding to the multiset $\{b_{i,1}, \dots, b_{i,q}\}$ (we delete the the zero entries and arrange the remaining positive entries in weakly decreasing order).
Now, consider the hypergraph $\mathcal{G}(\boldsymbol{\lambda}, \mathbf m)$. For each $i \in \mathrm{supp}(\mathbf{m})$, the corresponding clique has size $\ell({\lambda}_i)$, and we fix an ordering on its vertices. For any choice of pairwise distinct colors $c_1, \dots, c_{\ell({\lambda}_i)} \in \{1, \dots, q\}$ satisfying
\begin{equation}\label{4re}
\lambda^1_i = b_{i, c_1}, \dots, \lambda^{\ell({\lambda}_i)}_i = b_{i, c_{\ell({\lambda}_i)}},
\end{equation}
we color the first vertex of the $i$-th clique by $c_1$, the second by $c_2$, and so on. This gives an ordinary vertex coloring of the hypergraph $\mathcal{G}(\boldsymbol{\lambda}, \mathbf m)$.
It is easy to see that the number of choices of colors satisfying \eqref{4re} and thus the number of ordinary colorings we get is exactly $\prod_{i\in \mathrm{supp}(\bold m)}\prod_{k=1}^{\infty}(d^{{\lambda}_i}_{k})!$,
where $d_k^{{\lambda}_i}$ denotes the multiplicity of the part $k$ in the partition ${\lambda}_i$.

\medskip
\noindent
Conversely, let $\boldsymbol{\lambda} = ({\lambda}_i)_{i\in  \textbf{[}n \textbf{]}} \in S(\mathbf{m})$, and fix an ordering of the vertices of the hypergraph $\mathcal{G}(\boldsymbol{\lambda}, \mathbf{m})$. Suppose $\Gamma_{\mathcal{G}(\boldsymbol{\lambda}, \mathbf{m})}$ is a vertex coloring of this hypergraph such that the vertices of the $i$-th clique receive colors $c_1, c_2, \dots, c_r$, in order. Then we associate to this coloring a marked multi-coloring of hypergraph $\mathcal{G}$, defined by:
\[
{_{\bold{m}}\Gamma^{\mathrm{mark}}_\mathcal{G}}(i) = \{c_1^{\lambda^1_i}, \dots, c_r^{\lambda^r_i}\}, \quad \text{where } {\lambda}_i = (\lambda^1_i \ge \cdots \ge \lambda^r_i > 0).
\]
If some parts $\lambda^j_i$ are equal, then permuting the corresponding colors $c_j$ among the equal parts in the $i^{\mathrm{th}}$-clique yields the same marked multi-coloring. Hence, each marked multi-coloring corresponds to exactly
$\prod_{i\in \mathrm{supp}(\mathbf{m})}\prod_{k = 1}^\infty \left(d^{{\lambda}_i}_{k}\right)!$
distinct ordinary vertex colorings of $\mathcal{G}(\boldsymbol{\lambda}, \mathbf{m})$. Summing over all tuples $\boldsymbol{\lambda} \in S(\mathbf{m})$ gives the expression for the marked chromatic polynomial 
${_{\bold{m}}\Pi^{\mathrm{mark}}_\mathcal{G}}(q)$ of the hypergraph $\mathcal{G}$, and this completes the proof.
\end{proof}

\begin{example}
Let $\mathcal{G}$ be a hypergraph with vertex set $\mathcal{V} = \{1, 2, 3, 4\}$ and edge set $\mathcal{E} = \{\{1,2,3\},\; \{3,4\}\}$. Let $\mathbf{m} = (2,1,1,2)$, and suppose that $\mathcal{V}^{\mathrm{sp}} = \{1\}$.
By the definition of a marked multi-coloring of $\mathcal{G}$, the color set ${_{\mathbf{m}}\Gamma^{\mathrm{mark}}_\mathcal{G}}(1)$ is a multiset of $\{1, \dots, q\}$ because $1 \in \mathcal{V}^{\mathrm{sp}}$, while ${_{\mathbf{m}}\Gamma^{\mathrm{mark}}_\mathcal{G}}(i)$ is an ordinary subset of $\{1, \dots, q\}$ for each $i \notin \mathcal{V}^{\mathrm{sp}}$. Furthermore, the sizes of the respective color sets are determined by $\mathbf{m}$, so that:
$$
|{_{\mathbf{m}}\Gamma^{\mathrm{mark}}_\mathcal{G}}(1)| = 2, \
|{_{\mathbf{m}}\Gamma^{\mathrm{mark}}_\mathcal{G}}(2)| = 1, \
|{_{\mathbf{m}}\Gamma^{\mathrm{mark}}_\mathcal{G}}(3)| = 1, \
|{_{\mathbf{m}}\Gamma^{\mathrm{mark}}_\mathcal{G}}(4)| = 2.
$$
Thus, any marked multi-coloring of $\mathcal{G}$ associated with $\mathbf{m}$ and using at most $q$ colors must fall into one of the following two cases:

\medskip
\noindent \textbf{Case 1.} \,
$
{_{\mathbf{m}}\Gamma^{\mathrm{mark}}_\mathcal{G}}(1) = \{a,a\}, \
{_{\mathbf{m}}\Gamma^{\mathrm{mark}}_\mathcal{G}}(2) = \{b\}, \
{_{\mathbf{m}}\Gamma^{\mathrm{mark}}_\mathcal{G}}(3) = \{c\}, \
{_{\mathbf{m}}\Gamma^{\mathrm{mark}}_\mathcal{G}}(4) = \{d,e\},
$ where $a, b, c, d, e \in \{1, \dots, q\},\ d \ne e.$
Then the corresponding partitions are:
\[
\boldsymbol{\lambda}_1 = (2), \quad \boldsymbol{\lambda}_2 = (1), \quad \boldsymbol{\lambda}_3 = (1), \quad \boldsymbol{\lambda}_4 = (1,1).
\]

\noindent
In the associated hypergraph $\mathcal{G}(\boldsymbol{\lambda}, \mathbf{m})$, each vertex $i$ is replaced by a clique of size $\ell({\lambda}_i)$. The resulting vertex set is
$
\{1_1, 2_1, 3_1, 4_1, 4_2\}
$ and the edge set is
$$
\{\{1_1, 2_1, 3_1\},\; \{3_1, 4_1\},\; \{3_1, 4_2\},\; \{4_1, 4_2\}\}.$$
There are two distinct colorings of this hypergraph that correspond to the given marked multi-coloring:
\[
\begin{aligned}
\tau_1&\colon 1_1 \mapsto a,\; 2_1 \mapsto b,\; 3_1 \mapsto c,\; 4_1 \mapsto d,\; 4_2 \mapsto e, \\
\tau_2&\colon 1_1 \mapsto a,\; 2_1 \mapsto b,\; 3_1 \mapsto c,\; 4_1 \mapsto e,\; 4_2 \mapsto d.
\end{aligned}
\]

\medskip
\noindent \textbf{Case 2.} The remaining case is
$$
{_{\mathbf{m}}\Gamma^{\mathrm{mark}}_\mathcal{G}}(1) = \{a,f\}, \
{_{\mathbf{m}}\Gamma^{\mathrm{mark}}_\mathcal{G}}(2) = \{b\},\
{_{\mathbf{m}}\Gamma^{\mathrm{mark}}_\mathcal{G}}(3) = \{c\}, \
{_{\mathbf{m}}\Gamma^{\mathrm{mark}}_\mathcal{G}}(4) = \{d,e\},
$$ where $a, b, c, d, e, f \in \{1, \dots, q\},\ a \ne f,\ d \ne e.$
Then the corresponding partitions are:
\[
\boldsymbol{\lambda}'_1 = (1,1), \quad \boldsymbol{\lambda}'_2 = (1), \quad \boldsymbol{\lambda}'_3 = (1), \quad \boldsymbol{\lambda}'_4 = (1,1).
\]
The associated hypergraph $\mathcal{G}(\boldsymbol{\lambda}', \mathbf{m})$ has the vertex set
$
\{1_1, 1_2, 2_1, 3_1, 4_1, 4_2\},$
and the edge set:
\[
\{\{1_1, 2_1, 3_1\},\; \{1_2, 2_1, 3_1\},\; \{1_1, 1_2\},\; \{3_1, 4_1\},\; \{3_1, 4_2\},\; \{4_1, 4_2\}\}.
\]
There are four distinct colorings corresponding to this marked multi-coloring:
\[
\begin{aligned}
\tau'_1 &: 1_1 \mapsto a,\; 1_2 \mapsto f,\; 2_1 \mapsto b,\; 3_1 \mapsto c,\; 4_1 \mapsto d,\; 4_2 \mapsto e, \\
\tau'_2 &: 1_1 \mapsto a,\; 1_2 \mapsto f,\; 2_1 \mapsto b,\; 3_1 \mapsto c,\; 4_1 \mapsto e,\; 4_2 \mapsto d, \\
\tau'_3 &: 1_1 \mapsto f,\; 1_2 \mapsto a,\; 2_1 \mapsto b,\; 3_1 \mapsto c,\; 4_1 \mapsto d,\; 4_2 \mapsto e, \\
\tau'_4 &: 1_1 \mapsto f,\; 1_2 \mapsto a,\; 2_1 \mapsto b,\; 3_1 \mapsto c,\; 4_1 \mapsto e,\; 4_2 \mapsto d.
\end{aligned}
\]
Hence, the marked chromatic polynomial of the hypergraph $\mathcal{G}$ is given by:
\[
{_{\mathbf{m}}\Pi^{\mathrm{mark}}_\mathcal{G}}(q) = \frac{1}{2} \Pi_{\mathcal{G}(\boldsymbol{\lambda}, \mathbf{m})}(q) + \frac{1}{4} \Pi_{\mathcal{G}(\boldsymbol{\lambda}', \mathbf{m})}(q) = \frac{q^2(q-1)^2(q^2 - 4)}{4}.
\]

\end{example}

\subsection{} 
 We can compute the marked chromatic polynomials
of chordal graphs very explicitly, see  \cite[Theorem 2, Section 3.4]{CKV2025}, using Proposition \ref{ordinarychormatic}. 
The following theorem is immediate from Proposition \ref{ordinarychormatic}
and \cite[Theorem 3]{BHV24}.
\begin{thm}(\cite[Theorem 2, Section 3.4]{CKV2025})\label{peomarked}
    Let $G$ be a chordal graph with the vertex set $\textbf{[}n \textbf{]}$, suppose that the ordering on the vertices $\textbf{[}n \textbf{]} = \{1, \ldots, n\}$ is a perfect elimination ordering of the vertices of $\mathcal{G}$. Then we have
     \begin{equation}
    {_{\mathbf{m}}\Pi^{\mathrm{mark}}_\mathcal{G}}(q)=\sum\limits_{\boldsymbol{\lambda}\in S(\bold m)}\prod _{j\in \mathrm{supp}(\bold{m})} \binom{q-b^{\boldsymbol{\lambda}}_{j}}{\ell(\boldsymbol{\lambda}_{j})} \frac{\ell(\boldsymbol{\lambda}_{j})!}{\prod_{k=1}^{\infty}(d^{\boldsymbol{\lambda}_{j}}_{k}!)},\ \ b^{\boldsymbol{\lambda}}_{j}=
    \sum\limits_{\substack{i<j \\ i\in N_{G}(j)}}\ell(\boldsymbol{\lambda}_{i})
        \end{equation}
        \end{thm}

\section{An application to independence systems}
\noindent
Let $\mathcal{G}$ be a simple hypergraph. Then it is easy to see that the set of independent subsets of $\mathcal{G}$, $I(\mathcal{G})$ satisfies the following conditions:
\begin{itemize}
  \item the empty set $\emptyset$ belongs to $I(\mathcal{G})$,
    \item if $B \in I(\mathcal{G})$ and $C \subseteq B$, then $C \in I(\mathcal{G})$.
\end{itemize}
    
\medskip
\noindent
Any collection of subsets $\mathcal{A} \subseteq P(\textbf{[}n \textbf{]})$ satisfying the above two conditions is called an \textit{independence system} or \textit{abstract simplicial complex}. It is easy to see that any independence system must be of the form $I(\mathcal{G})$ for some hypergraph $\mathcal{G}$. 
\begin{prop}\label{indsys}
Let $\mathcal{A} $ be a collection of subsets of $\textbf{[}n \textbf{]}$ that form an independence system, i.e., $\mathcal{A}$ satisfies the following conditions:
\begin{enumerate}
    \item the empty set $\emptyset$ belongs to $\mathcal{A}$,
    \item if $B \in \mathcal{A}$ and $C \subseteq B$, then $C \in \mathcal{A}$.
\end{enumerate}
Given such a collection $\mathcal{A} \subseteq P(\textbf{[}n \textbf{]})$, define the hypergraph $\mathcal{G_A}$ with vertex set $\mathcal V = \textbf{[}n \textbf{]}$ and edge set
\[
\mathcal{E} = \{A \subseteq \textbf{[}n \textbf{]} \mid A \notin \mathcal{A}, \ B \in \mathcal{A} \text{ for all } B \subsetneq A\}.
\]
Then,
    $\mathcal{G_A}$ is a hypergraph such that $\mathcal{A}= I(\mathcal{G_A})$. Moreover,  $\mathcal{A}$ satisfies the additional condition that $\{v\}\in \mathcal{A}$ for all $v\in \mathcal{V}$ if and only if  $\mathcal{G_A}$ is a simple hypergraph.   
\end{prop}
\begin{proof}
Let $A\in \mathcal{A}$. Suppose $A\notin I(\mathcal{G_A})$, then there exists $B\in \mathcal{E}$ such that $B\subseteq A.$
From the condition $(2)$, we get $B\in \mathcal{A}$ and by definition of $\mathcal{E}$ we get $B\notin \mathcal{A}$, a contradiction. Thus, we have
$\mathcal{A}\subseteq I(\mathcal{G_A})$. Suppose $\mathcal{A}\subsetneq I(\mathcal{G_A})$, then choose $A\in I(\mathcal{G}_{\mathcal{A}})$ such that $A$ has smallest possible cardinality and 
$A\notin \mathcal{A}. $ This would immediately imply that $A\in \mathcal{E}$ and again we get a contradiction. Hence, we have $ \mathcal{A} = I(\mathcal{G_A}). $ 

\medskip
\noindent
For the last part, assume that $\mathcal{A}$ satisfies the additional condition that $\{v\}\in \mathcal{A}$ for all $v\in \mathcal{V}$, then all $A\in \mathcal{E}$ must satisfy $|A|\ge 2$ and if $B\subseteq A$ and both $A, B\in \mathcal{E}$ then we have $B\in \mathcal{A}$ unless $B=A$, a contradiction. The converse part is easy to verify. 
\end{proof}

\subsection{}
We illustrate the above association with some examples.
\begin{example}
    Let $n=3$ and $\mathcal{A}= \{\emptyset, \{1\}, \{2\},\{3\}, \{1,2\},\{2,3\}\}$. Then $\mathcal{G}_{\mathcal A}=(\mathcal{V}, \mathcal{E})$ where 
    $\mathcal{V}=\{1,2,3\}$ and $\mathcal{E}=\{ e_1\}$ such that $e_1=\{1,3\}.$

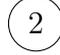
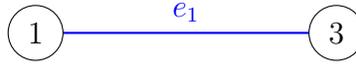
    \begin{center}
        \begin{tikzpicture}[scale=1, transform shape]
            \node[circle, draw, fill=white, minimum size=20pt] (v1) at (2,0) {1};
            \node[circle, draw, fill=white, minimum size=20pt] (v2) at (4,2) {2};
            \node[circle, draw, fill=white, minimum size=20pt] (v3) at (6,0) {3};

            \draw[thick, blue] (v1) -- (v3) node[midway, above] {$e_1$};
        \end{tikzpicture}
    \end{center}
\end{example}

\begin{example}
    Let $n=4$ and $\mathcal{A}= \{\emptyset, \{1\}, \{2\}, \{3\}, \{4\}, \{1,2\}, \{2,3\}, \{1,3\}\}$. Then, the hypergraph $\mathcal{G_A} = (\mathcal{V}, \mathcal{E})$ is defined as follows:
    \[
    \mathcal{V} = \{1,2,3,4\}, \quad \mathcal{E} = \{ e_1, e_2, e_3, e_4 \}
    \]
    where the hyperedges are:
    \[
    e_1 = \{1,4\}, \quad e_2 = \{2,4\}, \quad e_3 = \{3,4\}, \quad e_4 = \{1,2,3\}.
    \]

    \begin{center}
        \tikzset{every picture/.style={line width=0.75pt}} 
        \begin{tikzpicture}[x=0.75pt, y=0.75pt, yscale=-1, xscale=1]

            \draw  [color={rgb, 255:red, 74; green, 144; blue, 226}, draw opacity=1, line width=1.5]  
            (302.33,57.97) .. controls (306.75,57.77) and (310.4,61.19) .. (310.48,65.61) 
            -- (313.58,224.03) .. controls (313.67,228.45) and (310.16,232.19) .. (305.74,232.39) 
            -- (281.76,233.47) .. controls (277.35,233.66) and (273.7,230.24) .. (273.61,225.83) 
            -- (270.51,67.4) .. controls (270.43,62.98) and (273.94,59.24) .. (278.35,59.04) -- cycle;

            \draw [color={rgb, 255:red, 208; green, 2; blue, 27}, draw opacity=1, line width=1.5] (300,77) -- (387.45,141.92);
            \draw [color={rgb, 255:red, 126; green, 211; blue, 33}, draw opacity=1, line width=1.5] (299.45,141.92) -- (387.45,141.92);
            \draw [color={rgb, 255:red, 245; green, 166; blue, 35}, draw opacity=1, line width=1.5] (302.45,212.92) -- (387.45,141.92);

            \node[circle, draw, fill=white, minimum size=20pt] at (285,72) {1};
            \node[circle, draw, fill=white, minimum size=20pt] at (288,206) {3};
            \node[circle, draw, fill=white, minimum size=20pt] at (287,133) {2};
            \node[circle, draw, fill=white, minimum size=20pt] at (392,131) {4};

            \node [color={rgb, 255:red, 74; green, 144; blue, 226}] at (252,246) {$e_{4}$};
            \node [color={rgb, 255:red, 208; green, 2; blue, 27}, rotate=-32.59] at (350.38,89.19) {$e_{1}$};
            \node [color={rgb, 255:red, 126; green, 211; blue, 33}] at (333,125) {$e_{2}$};
            \node [color={rgb, 255:red, 245; green, 166; blue, 35}, rotate=-317.7] at (330.69,166.86) {$e_{3}$};

        \end{tikzpicture}
    \end{center}
\end{example}

\noindent
It is immediate from Proposition \ref{indsys}, we have an appealing combinatorial interpretation for the $I_S(\mathcal{A},x)^q[x^\mathbf{m}]$ for the independence system $\mathcal{A}$.
\begin{thm}
Let $\mathcal{A} $ be a collection of subsets of $\textbf{[}n \textbf{]}$ that form an independence system, and let $S\subseteq \textbf{[}n \textbf{]}$ be a special subset. Let $\mathcal{G}_\mathcal{A}$ be the hypergraph associated with $\mathcal{A}$ constructed as above and set $\mathcal{V}^{\mathrm{sp}}=S$. 
Then we have 
 $$I_S(\mathcal{A},x)^q[x^\mathbf{m}] =  {_{\bold{m}}\Pi^{\mathrm{mark}}_{\mathcal{G}_{\mathcal{A}}}}(q)$$

\end{thm}
\begin{proof}
    It is clear that $I_S(\mathcal{A},x)^q = I^{\mathrm{mark}}(\mathcal{G}_{\mathcal{A}}, x)^q$ with $\mathcal{V}^{\mathrm{sp}}=S$. 
\end{proof}

\section{Hyperplane arrangements and their characteristic polynomials}

\noindent
We will broadly follow \cite{Bjorner} for the theory of subspace arrangements.
 Let $\mathbb{R}^n$ be the $n$-dimensional vector space over $\mathbb R.$ For a subspace $H\subseteq \mathbb R^n$, we define $\mathrm{codim}(H) = n - \mathrm{dim}(H).$ For a subspace $H\subseteq \mathbb R^n$ given by the equations:
$$L^j(\mathbf x) = a_{1j}x_1+\cdots +a_{nj}x_n=0, 1\le j\le r;$$
we define the support of $H$ to be $\mathrm{supp}(H) = \{i\in  \textbf{[}n \textbf{]} : a_{ij} \neq 0 \ \text{for some $1\le j\le r$}\}.$
Note that the support of \(H\) depends on the choice of defining equations. Since we fix the defining equations throughout, this notation should not cause any ambiguity.

A \textit{subspace arrangement} or \textit{arrangement} is a finite collection of subspaces
$\mathscr{A} =\{H_1, \ldots, H_p\}$ of $\mathbb R^n.$ \textit{We only consider arrangements that are defined over $\mathbb{Z}$.}

We call an arrangement as a hyperplane arrangement if all $H_i$'s are hyperplanes, i.e., $\mathrm{codim}(H_i) = 1,\, 1\le i\le p$. 
Let the equations of the subspace $H_i$ be given by 
$L_i^j(\mathbf x) = a^i_{1j}x_1+\cdots +a^i_{nj}x_n=0$, $1\le j\le r_i$, \ $1\le i\le p$; then the product of the chosen defining linear forms is given by $Q(\mathbf x) = \prod_{i=1}^p \prod_{j=1}^{r_i} L_i^j(\mathbf x)$.
 The intersection poset of the arrangement $\mathscr{A}$ is denoted by $(L(\mathscr{A}), \leq )$, where 
 $$L(\mathscr{A}) = \{ H_{i_1}\cap \cdots \cap H_{i_k} : 1\le i_1<\cdots <i_k\le p\}\cup \{\mathbb R^n\}$$ and $\leq$ is defined as $x \leq y \text{ in }L(\mathscr{A}) \text{ if } x \supseteq y \text{ (as subsets of $\mathbb R^n$)}.$  The characteristic polynomial of $\mathscr{A}$ is defined by 
 $$\chi_{\mathscr{A}}(q) = \sum\limits_{X\in L(\mathscr{A})}\mu(\mathbb R^n, X) q^{\mathrm{dim}(X)},$$
 where $\mu$ is the \text{M\"obius} function of $(L(\mathscr{A}), \leq )$. Let $q$ be a prime power $\mathfrak{p}^\alpha$ and let $\mathscr{A}_q$ be the subspace arrangement in $\mathbb{F}_q$ induced from $\mathscr{A}$, precisely $\mathscr{A}_q =\{H_1^q, \ldots, H_p^q\}$ where 
 $$H_i^q = \{\mathbf x \in \mathbb F_q^n : L_i^j(\mathbf x) \equiv 0 \ (\mathrm{mod}\ \mathfrak{p}),\ 1\le j\le r_i\}, \ 1\le i\le p.$$
 It is well-known that (see \cite[Theorem 2.2]{Christos}, \cite[Theorem 5.15]{Stanley}) $L(\mathscr{A})\cong L(\mathscr{A}_q)$ for sufficiently large prime power $q$ and we also have that 
 $$\chi_{\mathscr{A}}(q)= \# \big(\mathbb F_q^n - \bigcup\limits_{H \in \mathscr{A}_q} H \big).$$

\begin{example}
\begin{enumerate}
    \item     Let $\mathcal{G}$ be a simple graph with vertex set $\textbf{[}n \textbf{]}$ and edge set $\mathcal{E}$. Then the 
    graphical arrangement of $\mathcal{G}$ is $\mathscr{A}_{\mathcal{G}} = \{H_e : e\in \mathcal{E}\},$
where $H_e = \{x_{i} = x_{j}\}$ if $e = \{i, j\}.$

    \item More generally, let $\mathcal{G}$ be a simple hypergraph with vertex set $\textbf{[}n \textbf{]}$ and edge set $\mathcal{E}$. Then the 
    hyper-graphical arrangement of $\mathcal{G}$ is $\mathscr{A}_{\mathcal{G}} = \{H_e : e\in \mathcal{E}\},$
where $H_e = \{x_{i_1} =\cdots = x_{i_r}\}$ if $e = \{i_1, \ldots, i_r\}.$
\end{enumerate}
\end{example}

\subsection{} Now we record the connection between the chromatic polynomials of hypergraphs and the characteristic polynomials of hyper-graphical arrangements. This is a classical result for graphical arrangement and it can be found, for example, in \cite[Theorem 2.7., Page no. 25]{Stanley}. Furthermore, an interpretation of the   characteristic polynomial of a subspace arrangement as the chromatic polynomial of an associated hypergraph is implicit in \cite[Theorem~3.4]{stanley1998graph}.

\begin{thm}Let $\mathcal{G}$ be a simple hypergraph with vertex set $\textbf{[}n \textbf{]}$ and edge set $\mathcal{E}$. Let
the hyper-graphical arrangement of $\mathcal{G}$ be given by $\mathscr{A}_{\mathcal{G}} = \{H_e : e\in \mathcal{E}\},$
where $H_e = \{x_{i_1} =\cdots = x_{i_r}\}$ if $e = \{i_1, \ldots, i_r\}.$ Then, we have
$$\chi_{\mathscr{A}_{\mathcal{G}}}(q) = \Pi_{\mathcal{G}}(q).$$
\end{thm}
\begin{proof}
    We know that for all sufficiently large prime powers $q >>0$ such that $L((\mathscr{A}_{\mathcal{G}})_q)\cong L(\mathscr{A}_{\mathcal{G}})$ and we also have that 
 $$\chi_{\mathscr{A}_{\mathcal{G}}}(q)= \# \big(\mathbb F_q^n - \bigcup\limits_{H \in (\mathscr{A}_{\mathcal{G}})_q} H \big).$$
But $\mathbb F_q^n - \bigcup\limits_{H \in (\mathscr{A}_{\mathcal{G}})_q} H $ consists of $(\alpha_1,\ldots, \alpha_n)\in \mathbb F_q^n$ such that for each edge $e\in \mathcal{E}$, there exists $i\neq j\in e$ such that 
$\alpha_i\neq \alpha_j$. That means it counts the number of functions $f: \textbf{[}n \textbf{]}\to \mathbb F_q$ such that, for each $e\in \mathcal{E}$, we have
not all $f(i)$ are equal, i.e.,
$\cap_{i\in e} \{f(i)\} =\emptyset$ (this is same as a vertex coloring of $\mathcal{G}$). That means
we have $$\# \big(\mathbb F_q^n - \bigcup\limits_{H \in (\mathscr{A}_{\mathcal{G}})_q} H \big) = \Pi_{\mathcal{G}}(q).$$ Hence we have $\chi_{\mathscr{A}_{\mathcal{G}}}(q) = \Pi_{\mathcal{G}}(q).$
 
\end{proof}

\subsection{}
Fix a set of special nodes $S\subseteq \textbf{[}n \textbf{]}$ as before. Motivated by the theory of graphs, we define the notion of $q-$coloring of $\mathscr{A}$, marked-chromatic polynomials of $\mathscr{A}$, and marked-independence series of $\mathscr{A}$
as follows: 
\begin{defn} Let $\mathbf{m}=(m_1, \ldots, m_n)$ be an $n$-tuple of non-negative integers.
\begin{enumerate}
    \item 
A $S$-marked or (simply marked) multi-coloring of $\mathscr{A}$ associated to $\mathbf{m}$ using at most $q$-colors 
 is a function
    $f: \textbf{[}n \textbf{]} \longrightarrow P^{\mathrm{mult}}(\mathbb F_q)$, given by $f(i)= C_i$, if 
     the following holds:
    \begin{enumerate}
        \item $C_i$ is a subset of $\mathbb{F}_q$ for all $i\notin S$,
        \item $| C_i|= m_i$ for all $1\le i\le n$,
        \item there is no $\mathbf x = (x_1, \ldots , x_n)$ such that $x_j \in C_j,$ $1 \leq j \leq n$, and  
        $\mathbf x\in H_i $ (i.e., the tuple $ (L_i^1(\mathbf x), \ldots, L_i^{r_i}(\mathbf x)) = 0$)  for some $1\le i\le p$.
    \end{enumerate} \medskip
    \item 
    We call two $q-$colorings $f, g$ of $\mathscr{A}$ distinct if $f(i)\neq g(i)$ for some $i\in \textbf{[}n \textbf{]}.$ \medskip 
    \item Define the marked-chromatic polynomial of $\mathscr{A}$ associated to $\mathbf{m}$ and $S$ by
     $$_{\mathbf m}\Pi^{\mathrm{mark}}_{\mathscr{A}}(q)= \# \text{distinct $S$-marked $q-$colorings of $\mathscr{A}$ associated with $\mathbf m$}.$$
     As before, we drop "mark" from $_{\mathbf m}\Pi^{\mathrm{mark}}_{\mathscr{A}}(q)$ if $S = \emptyset$, i.e., $C_i$'s are all subsets of $\mathbb F_q$ for all $1\le i\le n$.
   
     \end{enumerate}
\end{defn}

\begin{rem}
    It is important to note that, a~priori, the function counting the number of distinct $S$-marked $q$-colorings of $\mathscr{A}$ associated with $\mathbf{m}$ need not be a polynomial in~$q$. However, Theorem~\ref{mainthmarrangement} establishes a key result: this function $_{\mathbf{m}}\Pi^{\mathrm{mark}}_{\mathscr{A}}(q)$ is indeed a polynomial in~$q$. In particular, this allows us to evaluate it at any complex value of~$q$.
\end{rem}

\medskip
\noindent
The above remark allows us to define the marked-independence series of an arrangement.
\begin{defn} Let $q$ be a variable. The marked-independence series of $\mathscr{A}$ is defined by
         $$I^{\mathrm{mark}}_q(\mathscr{A},x)=\left(\sum_{\mathbf{m}\in \mathbb{Z}_{\ge 0}^n} {_{\bold{m}}\Pi^{\mathrm{mark}}_\mathscr{A}(q)}\ x^\mathbf{m}\right)^{1/q} = \exp\left(\frac{1}{q}\mathrm{log}\ \big( \sum_{\mathbf{m}\in \mathbb{Z}_{\ge 0}^n} {_{\bold{m}}\Pi^{\mathrm{mark}}_\mathscr{A}(q)}\ x^\mathbf{m}\big)\right)$$
Apriori the left hand side depends on $q$.
As before, when $S =\emptyset$, then we use simply $I_q(\mathscr{A},x)$ to denote $I^{\mathrm{mark}}_q(\mathscr{A},x)$. Note that, we have
$$I^{\mathrm{mark}}_q(\mathscr{A},x)^q = \sum_{\mathbf{m}\in \mathbb{Z}_{\ge 0}^n} {_{\bold{m}}\Pi^{\mathrm{mark}}_\mathscr{A}(q)}\ x^\mathbf{m}.$$
In particular, we can specialize $q$ to be any non-zero integer. We have the following immediate observation relating chromatic polynomials of
hyper-graphical arrangements and hypergraphs. 
\end{defn}

\begin{prop}
    Let $\mathcal{G}$ be a simple hypergraph with vertex set $\textbf{[}n \textbf{]}$ and edge set $\mathcal{E}$. Let
the hyper-graphical arrangement of $\mathcal{G}$ be given by $\mathscr{A}_{\mathcal{G}} = \{H_e : e\in \mathcal{E}\},$
where $$\text{
$H_e = \{x_{i_1} =\cdots = x_{i_r}\}$ if $e = \{i_1, \ldots, i_r\}.$}$$ Then, it is immediate that we have
$$_{\mathbf m}\Pi^{\mathrm{mark}}_{\mathscr{A}_{\mathcal{G}}}(q) = \ _{\mathbf m}\Pi^{\mathrm{mark}}_{\mathcal{G}}(q), \ \text{for all $S = \mathcal{V}^{\mathrm{sp}}\subseteq \mathcal{V}.$}$$ \qed
\end{prop}

\begin{example}\label{numberthyexample}
    Let us consider an arrangement $\mathscr{A}$ consists of a hyperplane $H$  whose equation given by $x_1+x_2+\cdots +x_n=x_{n+1}$ in $\mathbb R^{n+1}.$ Now given  $m_1,\cdots, m_{n+1}\ge 1$, the polynomial $_{\mathbf m}\Pi_{\mathscr{A}}(q)$ (for large prime $q>>0$) counts the
    number of ordered tuples  $(A_1,\ldots , A_{n+1})$ where $ A_1,\ldots , A_{n+1}\subseteq \mathbb F_q$ are non-empty subsets   satisfying:
    $$\text{ $|A_i|=m_i, 1\le i\le n+1$, and   $(A_1+\cdots +A_n)\cap A_{n+1} =\emptyset$}.$$
    Now using the Cauchy–Davenport theorem and induction, we get \text{(see \cite{Davenport})} $$\prod_{i=1}^{n}m_i\ge |\sum_{i=1}^n A_i|\ge \mathrm{min}\big(q, (\sum_{i=1}^nm_i)-n+1\big).$$ 
   Now, for any given choices of $A_i\subseteq \mathbb F_q$ having $|A_i|=m_i$, $1\le i\le n$, we must have 
   $|\mathbb{F}_q - (\sum_{i=1}^n A_i)| = q-(\sum_{i=1}^nm_i)+n-r-1$ for some $r\ge 0.$ 
   We can choose any subset
    from $\mathbb{F}_q - (\sum_{i=1}^n A_i)$ having $m_{n+1}$ elements for $A_{n+1}$, however we will  have $\{A_i\}_{i=1}^n$ and $\{B_i\}_{i=1}^n$ such that $\sum_{i=1}^n A_i =\sum_{i=1}^n B_i$. So we need to be careful about this. Considering this, we get
    \begin{equation}\label{examplefirst}
        _{\mathbf m}\Pi_{\mathscr{A}}(q) = \sum\limits_{r=0}^{M}\sum\limits_{C\subseteq \mathbb F_q \atop |C|=\sum_{i=1}^nm_i-n+r+1} \alpha_C(q)
{q-(\sum_{i=1}^nm_i)+n-r-1\choose m_{n+1}},
    \end{equation}
 where $M = \prod_{i=1}^{n}m_i-(\sum_{i=1}^nm_i)+n-1$ and for a given $C\subseteq \mathbb F_q$, 
 $\alpha_C(q)$ counts the number of ordered tuples  $(A_1,\ldots , A_{n})$ where  $A_1,\ldots , A_{n}\subseteq \mathbb F_q$ are non-empty subsets
 such that $|A_i|=m_i, 1\le i\le n$, and $\sum_{i=1}^n A_i = C$.
It is clear that we get (by over counting)
    $$_{\mathbf m}\Pi_{\mathscr{A}}(q) \le \sum\limits_{r=0}^{M}{q\choose m_1}{q\choose m_2}\cdots {q\choose m_n}{q-(\sum_{i=1}^nm_i)+n-r-1\choose m_{n+1}}.$$
\end{example}

\begin{rem}
    We emphasize here that there are no obvious reasons to believe that the right hand side of the equation \ref{examplefirst} gives us a polynomial in $q$. It is a polynomial in $q$ follows from Theorem \ref{mainthmarrangement}. One could compute these polynomials explicitly using Vosper's theorem (see \cite{Vosper}) etc. It will be studied elsewhere. 
\end{rem}

\medskip
\subsection{} Given a subspace arrangement $\mathscr{A}$, set of special nodes $S$, and  $\mathbf{m}\in \mathbb Z_{\ge 0}$, we introduce a new arrangement;
we denote it by $\mathscr{A}(S, \mathbf{m})$, and call it  as $S$-marked (or simply marked) $\mathbf{m}$-clan of $\mathscr{A}.$
The  $S$-marked $\mathbf{m}$-clan of $\mathscr{A}$ is an arrangement in $\mathbb R^{|\mathbf{m}|}$ consisting of the following subspaces:
\begin{itemize}
    \item for each $i\in S\cap \mathrm{supp}(\mathbf{m})$ and $1\le r\neq s \le m_i$, we have $H(i, r, s)$ is given by the equation
    $P_{rs}^i(\mathbf x)=x_{ir}-x_{is} =0$. 

\medskip

    \item For every tuple $\mathbf{k} = (k_i)_{i \in \mathrm{supp}(\mathbf{m})}$ with $1 \le k_i \le m_i$ for all $i \in \mathrm{supp}(\mathbf{m})$, and 
    for all $1 \le s \le p$ such that $\mathrm{supp}(H_s) \subseteq \mathrm{supp}(\mathbf{m})$, we have the subspace $H_{s\mathbf{k}}$ given by the equations
\[
L_{s\mathbf{k}}^j(\mathbf{x}) = \sum_{i\in \mathrm{supp}{(H_s)}}a^s_{ij} x_{ik_i} = 0; \ \text{for all $1 \le j \le r_s.$}
\]

\end{itemize} \noindent
Precisely, for each in $i\in \mathrm{supp}(\mathbf{m})$, we introduce $m_i$ number of new variables $x_{ir}$, \ $1\le r\le m_i$, and  
we do not introduce any new variable for $i\in \textbf{[}n \textbf{]}$ such that $m_i =0$. As before, we simply use $\mathscr{A}(\mathbf{m})$ for 
$\mathscr{A}(S, \mathbf{m})$ if $S =\mathrm{supp}(\mathbf{m}).$

\begin{example}
    Let $\mathscr{A} = \{H = \{x_1-x_2+x_3-x_4 = 0\}\}$ and $\mathbf{m} = (2, 2, 1, 2)$. Then, we have $\mathrm{supp}(\mathbf{m}) =  \textbf{[}4 \textbf{]}.$
    \begin{enumerate}
        \item If $S = \textbf{[}4 \textbf{]}$ then $\mathscr{A}(\mathbf{m})$ consisting of the following subspaces whose equations give by: 
    \begin{itemize}
    \item $x_{11}-x_{12} = 0$; $x_{21}-x_{22} = 0$; and $x_{41}-x_{42} = 0$; 
    \item $x_{1{k_1}}-x_{2k_2}+x_3-x_{4k_4} = 0$; $ 1\le k_r \le m_r; r = 1, 2, 4.$
    \end{itemize}
\medskip
        \item If $S =\emptyset$ then $\mathscr{A}(\emptyset, \mathbf{m})$ consisting of the subspaces given by:
    \begin{itemize}
    \item $x_{1{k_1}}-x_{2k_2}+x_3-x_{4k_4} = 0$, $ 1\le k_r \le m_r, r = 1, 2, 4.$
    \end{itemize}

     \end{enumerate}
\end{example}
\begin{example}
     Let $\mathscr{A} = \{H_1 = \{x_1+x_2+x_3+x_4 = 0\}, H_2=\{x_1=x_3\}\}, $ and $\mathbf{m} = (2, 0, 1, 0)$. Then, we have $\mathrm{supp}(\mathbf{m}) =  \{1, 3 \}.$
   If $S = \{1, 3 \}$ then $\mathscr{A}(S, \mathbf{m})$ consisting of the following subspaces whose equations are given by: 
    \begin{itemize}
    \item $x_{11}-x_{12} = 0$;
    \item $x_{1{1}}=x_{31}; x_{1{2}}=x_{31}$.
    \end{itemize}
\end{example}

\subsection{}\label{mainthmarrangementsection} As before, for given $\bold m\in \mathbb{Z}_{\geq 0}$, we define
 $$S(\mathbf m)=\{ \boldsymbol{\lambda}=({\lambda}_i)_{i \in  \textbf{[}n \textbf{]}}: {\lambda}_i \text{ is a partition of }m_i \text{ and } {\lambda}_i=(1^{m_i})\ \forall i\notin S\}.$$  Set $ m(S) = \prod_{i\in S\cap \mathrm{supp}(\mathbf{m})} m_i!.$
 Let $\mathscr{A}$ be a given subspace arrangement with the set of special nodes $S$.
 For each tuple $\boldsymbol{\lambda}\in S(\bold m)$, we associate a new arrangement 
 $\mathscr{A}(\boldsymbol{\lambda},\bold m)$ of  $\mathbb{R}^{|\boldsymbol{\lambda}|}$, where $|{{\boldsymbol{\lambda}}|}=\sum_{i\in [n]}\ell({\lambda}_i)$, as follows: 
 \medskip
 \begin{itemize}
    \item For each $i\in \mathrm{supp}(\mathbf{m})$ and $1\le r\neq s \le \ell({\lambda}_i)$, we have $H(i, r, s)$ is given by the equation
    $P_{rs}^i(\mathbf x)=x_{ir}-x_{is} =0$. 

\medskip

    \item For every tuple $\mathbf{k} = (k_i)_{i \in \mathrm{supp}(\mathbf{m})}$ with $1 \le k_i \le \ell({\lambda}_i)$ for all $i \in \mathrm{supp}(\mathbf{m})$, and 
    for all $1 \le s \le p$ such that $\mathrm{supp}(H_s) \subseteq \mathrm{supp}(\mathbf{m})$, we have the subspace $H_{s\mathbf{k}}$ given by the equations
\[
L_{s\mathbf{k}}^j(\mathbf{x}) = a^s_{1j} x_{1k_1} + \cdots + a^s_{nj} x_{nk_n} = 0; \ \text{for all $1 \le j \le r_s.$}
\]

\end{itemize}
We have the following close relationship between the marked chromatic polynomials  and the characteristic polynomials. 
\begin{thm}\label{mainthmarrangement}
    Let $\mathscr{A}$ be an arrangement and  $\mathbf m = (m_1, \ldots, m_n)\in  \mathbb{Z}_{\ge 0}^n$, and let $S \subseteq \mathrm{supp}(\mathbf{m})\subseteq  \textbf{[}n \textbf{]}$ be the set of special nodes. Then we have 
   \begin{equation}\label{equarrangement}
{}_{\mathbf{m}}\Pi^{\mathrm{mark}}_{\mathscr{A}}(q) = \sum_{\boldsymbol{\lambda}\in S(\bold m)}
\frac{\chi_{\mathscr{A}(\boldsymbol{\lambda}, \mathbf{m})}(q)}
{\prod_{i\in \mathrm{supp}(\bold m)}\prod_{k=1}^{\infty}(d^{{\lambda}_i}_{k}!)}.
\end{equation}
In particular, ${}_{\mathbf{m}}\Pi^{\mathrm{mark}}_{\mathscr{A}}(q)$ is a polynomial in $q,$ and we can specialize $q$ to any complex number. 
\end{thm}
\begin{proof}
To each element $\boldsymbol{\lambda} = ({\lambda}_i)_{i \in I} \in S(\mathbf{m})$, we associate $\mathbb{M}_{\mathscr{A}(\boldsymbol{\lambda}, \mathbf{m})}$ which is the collection of
$\mathbf{x}\in \mathbb F_q^{|{\boldsymbol{\lambda}}|}$ such that \medskip
\begin{itemize}
    \item $ x_{ir}\neq x_{is}$ for all $1\le r\neq s \le l({\lambda}_i),\ i\in \mathrm{supp}(\mathbf{m});$
\medskip
    \item $(x_{1k_1},\ldots, x_{n{k_n}})\notin H_{s\mathbf{k}} $
    for all $ \mathbf{k}\in \{(k_i)_{i\in \mathrm{supp}(\mathbf{m})} : 1\le k_i \le l({\lambda}_i), \ \forall i\in \mathrm{supp}(\mathbf{m})\}$ and
   for all $1\le s\le p $ such that $\mathrm{supp}(H_s)\subseteq \mathrm{supp}(\mathbf{m})$. \end{itemize} \medskip
It is easy to see that $\chi_{\mathscr{A}(\boldsymbol{\lambda}, \mathbf{m})}(q)$ counts the 
number of elements in $\mathbb{M}_{\mathscr{A}{(\boldsymbol{\lambda}, \mathbf{m})}}$ for sufficiently large good prime powers. 
For a given $\mathbf{x}\in \mathbb{M}_{\mathscr{A}{(\boldsymbol{\lambda}, \mathbf{m})}}$, define the function $f_{\mathbf{x}}: \textbf{[}n \textbf{]} \longrightarrow P^{\mathrm{mult}}(\mathbb F_q)$, given by $f_\mathbf{x}(i)= C_i$, where
$$\text{
$C_i = \{x^{\boldsymbol{\lambda}^\alpha_i}_{i\alpha} : 1\le \alpha\le \ell({\lambda}_i)\}$, where ${\lambda}_i = (\lambda^1_i \ge \cdots \ge \lambda^{r_i}_i > 0)$.}$$ Then, $f_{\mathbf{x}}$ gives us a $S$-marked multi-coloring of $\mathscr{A}$ associated to $\mathbf{m}$ using at most $q$-colors.
If some parts $\boldsymbol{\lambda}^\alpha_j$ are equal, then permuting the corresponding  $x_{j\alpha}$ among the equal parts in  the co-ordinates $x_{j1}, \ldots , x_{j\ell(\boldsymbol{\lambda}_j)}$  yields the same S-marked multi-coloring. Conversely if $f: \textbf{[}n \textbf{]} \longrightarrow P^{\mathrm{mult}}(\mathbb F_q)$ be  a $S$-marked multi-coloring of $\mathscr{A}$ associated to $\mathbf{m}$ using at most $q$-colors given by $f(i)=C_i$ where $C_i=\{x^{\boldsymbol{\lambda}^\alpha_i}_{i\alpha} : 1\le \alpha\le \ell({\lambda}_i)\}$ gives $\prod_{i\in \mathrm{supp}(\bold m)}\prod_{k=1}^{\infty}(d^{{\lambda}_i}_{k}!)$ number of $\mathbf{x}\in \mathbb{M}_{\mathscr{A}(\boldsymbol{\lambda}, \mathbf{m})}$. Therefore we have $\frac{\chi_{\mathscr{A}(\boldsymbol{\lambda}, \mathbf{m})}(q)}
{\prod_{i\in \mathrm{supp}(\bold m)}\prod_{k=1}^{\infty}(d^{{\lambda}_i}_{k}!)}$ number of S-marked multi-coloring of $\mathscr{A}$ corresponding to  $\boldsymbol{\lambda} = ({\lambda}_i)_{i \in I} \in S(\mathbf{m})$. Summing over all tuples $\boldsymbol{\lambda} \in S(\mathbf{m})$ gives the expression for the number of S-marked multi-coloring ${_{\bold{m}}\Pi^{\mathrm{mark}}_\mathscr{A}}(q)$ of $\mathscr{A}$, and this completes the proof.

\end{proof}

\subsection{} We have the following immediate corollaries of Theorem \ref{mainthmarrangement}.
\begin{cor}
    Let $\mathscr{A}$ be an arrangement and  $\mathbf m = (m_1, \ldots, m_n)\in  \mathbb{Z}_{\ge 0}^n$. Then we have 
   \begin{equation}\label{equarrangementcor}
{}_{\mathbf{m}}\Pi_{\mathscr{A}}(q) = 
\frac{\chi_{\mathscr{A}(\mathbf{m})}(q)}
{m(\mathrm{supp}(\mathbf{m}))}.
\end{equation} \qed
\end{cor}

\medskip
\noindent
Note that, $(-1)^{|\mathbf{m}|}\chi_{\mathscr{A}(\mathbf m)}(-1)$
counts $\#$ connected components of
$\mathbb{R}^{|\mathbf{m}|} - \bigcup\limits_{H\in \mathscr{A}(\mathbf{m})}H$ by \cite{Zaslavsky}. This motivates us to study the coefficients of $I(\mathscr{A}, -x)^{-q} $. 
Indeed, we have:
\begin{cor}
    Let $\mathscr{A}$ be a hyperplane arrangement in $\mathbb R^n$. Then we have $$ \text{$I_{-q}(\mathscr{A}, -x)^{-q} \ge 0$ for all $q\in \mathbb{Z}_{\ge 1}.$}$$
\end{cor}
\begin{proof} Fix $q\ge 1.$
Note that $I_q(\mathscr{A}, -x)^q$ is obtained by replacing $x_i$ by $-x_i$ for all $i\in \textbf{[}n \textbf{]}$ in $I_q(\mathscr{A}, x)^q$. 
Thus we have   $I_{-q(}\mathscr{A}, -x)^{-q}=\sum_{\mathbf{m}\in \mathbb{Z}_{\ge 0}^n}(-1)^{|\mathbf{m}|} {_{\bold{m}}\Pi_\mathscr{A}(-q)}\ x^\mathbf{m}$, we only need to
 prove that     $$(-1)^{|\mathbf{m}|} \,_{\mathbf m}\Pi_{\mathscr{A}}(-q)= (-1)^{|\mathbf{m}|} \frac{\chi_{\mathscr{A}(\mathbf m)}(-q)}{\mathbf m!}\ge 0$$
for all $\mathbf{m}\ge 0$.  This is immediate because the numerator of the right hand side $$(-1)^{|\mathbf{m}|}\chi_{\mathscr{A}(\mathbf m)}(-q) = \sum\limits_{X\in \mathcal{L}(\mathscr{A}(\mathbf{m}))} (-1)^{\mathrm{codim}(X)}\mu(\mathbb R^{|\mathbf{m}|}, X)q^{\mathrm{dim}(X)}.$$
Since $\mathcal{L}(\mathscr{A}(\mathbf{m}))$ is a geometric lattice (see \cite[Proposition 3.6; Page no. 35 and Theorem 3.10; Page no. 37]{Stanley}), we have  that $$\text{$(-1)^{\mathrm{codim}(X)}\mu(\mathbb R^{|\mathbf{m}|}, X) \ge 0$ for all $X\in \mathcal{L}(\mathscr{A}(\mathbf{m})).$} $$

\end{proof}

\begin{cor}\label{graphnonnegative}
    Let $G$ be a simple graph. Then we have $I(G, -x)^{-q} \ge 0$ for all $q\ge 1.$
\end{cor}

\begin{rem}
    Several proofs can be given to Corollary \ref{graphnonnegative}. Our proof uses the fact that the chromatic polynomials of graphs are same as the characteristic polynomials of graphical hyperplane arrangements. 
\end{rem}

\section{Non-negativity of $(-q)$-power of independence polynomial of hypergraphs}
\noindent
Let $\mathcal{G}$ be a simple hypergraph. Consider the following series
$$I(\mathcal{G}, -x)^q = \sum\limits_{\mathbf{m}\ge 0} (-1)^{|\mathbf{m}|}{_{\mathbf{m}}\Pi_{\mathcal{G}}}(q) x^{\mathbf{m}},$$
which is obtained by replacing $x_i$ by $-x_i$ for all $i\in \textbf{[}n \textbf{]}$. We want to analyze when $I(\mathcal{G}, -x)^{-q}\ge 0$ for all $q\in \mathbb Z_{\ge 1}.$
First observe that, we have $I(\mathcal{G}, -x)^{-q}\ge 0$ for all $q\in \mathbb Z_{\ge 1}$ if and only if $I(\mathcal{G}, -x)^{-1}\ge 0$.
So it is enough to study when $I(\mathcal{G}, -x)^{-1}$ becomes non-negative.  We immediately have the necessary condition.
\begin{prop}\label{nonnegative}
   Let $\mathcal{G}$ be a simple hypergraph.  If $I(\mathcal{G}, -x)^{-1} \ge 0$ then $\mathcal{G}$ is an even hypergraph (i.e., all edges of $\mathcal{G}$ must have even number of elements). 
\end{prop}
\begin{proof}
First note that if  $I(\mathcal{G}, -x)^{-1} \ge 0$ then  $I(\mathcal{G}, -x)^{-1}|_{x_i = 0; \ i\in I} \ge 0$  for any subset $I\subseteq \mathcal{V}$.
Suppose $\mathcal{G}$ is not an even hypergraph, then there exists $e\in \mathcal{E}$ such that $|e| $ is odd number. Write $e = \{i_1,\ldots, i_r\}$, $r\ge 3$.
Let $\widehat{\mathcal{G}}_e$ be the hypergraph obtained by taking vertices $\widehat{\mathcal{V}}_e = \{i_1,\ldots, i_r\}$ and edges as $\widehat{\mathcal{E}}_e = \{e\},$ (i.e.,  $\widehat{\mathcal{G}}_e$ is obtained by removing all vertices in $\mathcal{V}\backslash \{i_1,\ldots, i_r\}$).
Now set $x_i = 0$ for all $i\notin e$ in $f = I(\mathcal{G}, -x)^{-1}$, then it is easy to see that we have $$f|_{x_i = 0;\ i\notin e} = I(\widehat{\mathcal{G}}_e , -x)^{-1}.$$
   From Example \ref{impexample}, for $\mathrm{2} = (2, \ldots, 2)$, we have $$I(\widehat{\mathcal{G}}_e , -x)^{-1}[x^{\mathbf{2}}] =\left(\prod_{i\in e}(1-x_i)+\prod_{i\in e}x_i \right)^{-1}[x^{\mathbf{2}}]  = 2+ (-2)^r<0, \ \text{where $\mathrm{2} = (2, \ldots, 2)$,}$$ which leads to a contradiction.
\end{proof}

\subsection{} The converse of the Proposition \ref{nonnegative} is not true. For example, consider the hypergraph $\mathcal{G}$ with vertex set
$\mathcal{V}=\{1,2,3,4,5,6\}$
and edge set
\[
\mathcal{E}
=
\bigl\{
E_1,E_2
\bigr\},
\qquad
E_1=\{1,2,3,5\},
\qquad
E_2=\{2,4,5,6\}.
\]
Both edges have cardinality four, and neither edge contains the other.
Thus, $\mathcal{G}$ is a simple even hypergraph.
Set
\[
P(\mathbf{x})=\prod_{i=1}^{6}(1-x_i).
\]
Then we have
\begin{align*}
I(\mathcal{G},-\mathbf{x})
={}&
P(\mathbf{x})
-x_1x_2x_3x_5(1-x_4)(1-x_6)\\
&-x_2x_4x_5x_6(1-x_1)(1-x_3)
+x_1x_2x_3x_4x_5x_6.
\end{align*} 
It is easy to see that
\begin{align*}
&[x_1^2x_2x_3^2x_4^2x_5x_6]\,
I(\mathcal{G},-\mathbf{x})^{-1} =
-1.
\end{align*}
Therefore,
the condition that every edge has even cardinality is necessary but not
sufficient for the coefficientwise non-negativity of
$I(\mathcal{G},-\mathbf{x})^{-1}$.
It is an interesting problem to determine for which hypergraph $\mathcal{G}$, the power series $I(\mathcal{G},-\mathbf{x})^{-1}$ in non-negtive. 

\bibliographystyle{plain}
\bibliography{bibliography}
\end{document}